\documentclass{elsart}%
\usepackage{amsmath}
\usepackage{mathtools}
\usepackage{graphicx}
\usepackage{float}
\usepackage{indentfirst}
\usepackage{verbatim}
\usepackage{algorithm}
\usepackage{algpseudocode}
\usepackage{float}
\usepackage{multicol}
\usepackage{listings}
\usepackage{subfigure}
\usepackage{epsfig}
\usepackage{amsfonts}
\usepackage{amssymb}
\usepackage{epstopdf}
\usepackage{color}
\usepackage{mathrsfs}
\usepackage{epstopdf}%
\newcommand{\N}{\mathbb N}
\newcommand{\C}{\mathbb C}
\newcommand{\R}{\mathbb R}

\newcommand{\K}{\mathbb K}
\newenvironment{proof}
{\textit{Proof:} }
{\hfill$\square$}
\newtheorem{definition}{Definition}[section]

\newtheorem{theorem}{Theorem}

\begin{document}
\begin{frontmatter}
\title{A hierarchical random compression method for kernel matrices}
\author[UNCC]{Duan Chen}
\author[SMU]{Wei Cai}
\address[UNCC]{Department of Mathematics and Statistics, University of
North Carolina at Charlotte, Charlotte, NC 28223, USA}
\address[SMU]{Department of Mathematics, Southern Methodist University, Dallas, TX75275}

\bigskip
{\bf Suggested Running Head:}
\\
Hierarchical random compression method
\\
\bigskip
{\bf Corresponding Author: }
\\
Prof. Wei Cai \\
Department of Mathematics\\
Southern Methodist University \\
Dallas, TX75275 \\
Phone: 214-768-3320 \\
Email: cai@smu.edu
\newpage
\begin{abstract}
In this paper, we propose a hierarchical random compression method (HRCM)
for kernel matrices in fast kernel summations.
The HRCM combines the hierarchical framework of the $\mathscr{H}$-matrix and  a randomized sampling technique of
the column and row spaces for far-field interaction kernel matrices. We show
that a uniform column/row sampling (with a given sample size) of a far-field kernel matrix, without the need and associated cost
to pre-compute a costly sampling distribution, will give a low-rank
compression of such low-rank matrices, independent of the matrix sizes and only dependent on the separation of
the source and target locations. This far-field random compression technique is then implemented at each level of the
hierarchical decomposition for general kernel matrices, resulting in an $O(N\log N)$ random compression
method. Error and complexity analysis for the HRCM are included. Numerical results for electrostatic and
Helmholtz wave kernels have validated the efficiency and accuracy of the proposed method with a cross-over
matrix size, in comparison of direct $O(N^2)$ summations, in the order of thousands for a 3-4 digits relative accuracy.
\end{abstract}
\end{frontmatter}


\section{Introduction}

Kernel matrices arises from many scientific and engineering computation,
data analytics, and deep learning algorithms. How to efficiently handle those matrices
has been an active research
area as the size of the matrices increase dramatically due to large dimension
data set in the era of exascale computing and big data.
In many fields of applications, kernel matrices are used for summations of the following form
\begin{equation}
E_{i}=\sum_{j=1}^{N}\mathcal{K}(\mathbf{r}_{i},\mathbf{r}_{j})q_{j}, \quad i = 1,2,...M,
\label{eqn:original}%
\end{equation}
where $\mathcal{K}(\mathbf{r}_{i},\mathbf{r}_{j})$ is the kernel function
representing the interaction betweens $M$ targets  and $N$ sources, which have position coordinates $\mathbf{r}_{i}$, $\mathbf{r}_{j}$ and density $q_{i}$, $q_{j}$, respectively. In scientific computing, Eq. (\ref{eqn:original}) arises from boundary element discretization for partial differential equations such as  Poisson-Boltzmann equation, Helmholtz equation, Maxwell equations or fractional differential equations. The kernel summation also plays a fundamental role in non-parametric statistics or machine learning techniques in applications such as Latent semantic indexing (LSI), or analysis in DNA microarray data, or eigenfaces and facial recognition.


When $M$ and $N$ are large, the summation task as Eq. (\ref{eqn:original}) becomes prohibitively expensive for practical computations. To speed up the kernel matric summation, a low rank approximation usually is sought to reduce the operation from
$O(MN)$ to $O(KN)$ in a $K$-rank approximation where $K\ll N$. Generally, a small subspace, containing the major action of the original matrix, is first identified, which is then compressed and
approximated by a low rank representation. Traditional truncated singular value decomposition (SVD) and
rank-revealing QR factorization are popular ways to construct such low rank
approximations. But the cost of implementing SVD or QR themselves are much higher than a single matrix-vector multiplication, which is the typical core operation
required for an iterative solution of the linear systems related to the kernel matrices.
The fast multipole method (FMM) \cite{greengard1997fast} is one of the most important fast algorithms for kernel summation, in which target and source points are hierarchically divided as well-separated sets, and on each set, the kernel function is low-rank approximated by using multipole expansions. In the original FMM, kernel function is approximated by analytical tools (either with addition theorems of special functions or Taylor expansions) \cite{greengard1997fast,chew2001fast,darve2000fast,gimbutas2003generalized,cai2018}. To overcome the difficulties when analytic formulation of kernel functions is not available, various semi-analytic\cite{anderson1992implementation,gray2001n,lee2014distributed} and algebraic FMMs \cite{martinsson2011randomized,yarvin1998generalized,ying2004kernel} were developed in recent decades. In some other approaches \cite{kapur1998n,kapur1997fast}, the whole kernel matrix is split into block matrices with various ranks, on each of which the SVD was implemented and then a truncated summation was used. More recently, a random interpolative decomposition method was also developed in the framework of FMM \cite{march2015askit,march2017far,yu2017geometry}.

To take advantage of modern computational architecture and especially for high dimensional data, randomized methods are recognized to provide efficient approach of low-rank
approximations  to handle transient data sets, where access to the full matrices may not be
possible or too expensive as only single pass or constant number of passes of the
matrices are preferred or realistic.
Some fast Monte-Carlo algorithms for large matrix-matrix multiplications and low-rank approximations have been developed in \cite{drineas2006fast1,drineas2006fast2}. Still, those algorithms are generally not more efficient than {\em a single} implementation of Eq. (\ref{eqn:original}).
For example,  it was proposed \cite{frieze2004}\cite{siamrv2011} that by randomly sampling the column and rows of a matrix, based on a the magnitude of the row/column vector, will give a good low rank approximation in high probability. However, such an approach requires the calculation of the $L_2$ norm of the column/row vectors, which is already of complexity $O(MN)$  as the direct kernel summation cost. Furthermore, error analysis in \cite{drineas2006fast1,drineas2006fast2} is for general matrices and relies on large number of samples. Number of samples will be greatly limited in practice if high efficiency is pursued and the low-rank property of the matrix should give some additional benefits in error analysis.

 In this paper, we develop a novel hierarchical random compression algorithm for kernel matrices in order to enhance the efficiency of kernel summation at a very large scale.
 First we will apply the randomized sampling of the column and
row space technique for the kernel matrix resulting from far-field interactions,
i.e.  the target and source points set are well-separated. We show that for such a
scenery, the uniform sampling distribution will be sufficient to give a good
low-rank approximation, thus removing the need and associated cost of computing sampling distributions. The expectation of error depends on the number of sampled column (row) and  the diameter-distance ratio of the two sets.
For general source-field configurations, we will employe the $\mathscr{H}$-matrix framework \cite{hackbusch1999sparse,hackbusch2000sparse} to  construct a hierarchical multilevel
tree structure, and apply the far-field randomized low-rank approximation for admissible interactions. It should be noted that this hierarchical approach is different from that of FMM methods; it requires only a one-way top-down pass, and is simple to implement in a recursive manner.

The rest of the paper is organized as follows.  Section 2 describes the randomized compression method
for far field interaction matrices. Analysis of the algorithm for the far field
case is given in Section 3. Section 4 introduces the hierarchical random
compression method ((HRCM) where $\mathscr{H}$-matrix framework is used with the randomized
compression method for far field on different hierarchical level of the data
set. Numerical tests for the proposed HRCM are provided for kernels from
electrostatic and wave Green's functions in Section 5. Finally, conclusion and
discussion are given in Section 6.


\section{Basic random algorithms for well-separated sources and targets}

For convenience, we list some basic  notations in linear algebra.  For a vector ${\bf x}\in\K^{N}$, where $\K$ represents either real $\R$ or complex $\C$,  denote Euclidean length by
\begin{equation}
|{\bf x}|=\left(  \sum_{i=1}^{N}|x_{i}|^{2}\right)  ^{1/2}.%
\end{equation}
For a matrix $\mathbf{A}\in\K^{M\times N}$, let $A^{(j)}$ and
$A_{(i)}$ denote its $j$-th column and $i$-th row,
respectively. The Frobenius norm is
\begin{equation}
\|\mathbf{A}\|_{F}=\sqrt{\sum_{i=1}^{m}\sum_{j=1}^{n} A_{ij}^{2}},%
\end{equation}
and the product of $\mathbf{AB}$ can be written as
\begin{equation}
\mathbf{AB} = \sum_{j=1}^{N}A^{(j)}B_{(j)}.
\end{equation}

In this section and Section \ref{sec:analysis} we restrict ourselves to low-rank matrices corresponding to the approximation of kernel function for well-separated target and source points.

Define the diameter
of the target set $\{{\bf r}_i\}$ as
\begin{equation}
\label{eqn:diam}diam(T) := \max_{i,i^{\prime}}|\mathbf{r}_{i}-\mathbf{r}%
_{i^{\prime}}|,
\end{equation}
using the Euclidean norm in $\R^{d}$. Diameter of the source set $diam(S)$ is defined similarly. Additionally, we will need the distance of
the two groups as
\begin{equation}
\label{eqn:dist}dist(T,S) := |\mathbf{r}_{T}^{*}- \mathbf{r}_{S}^{*}|,
\end{equation}
where $\mathbf{r}_{T}^{*}$ and $\mathbf{r}_{S}^{*}$ are the Chebyshev centers
of sets $T$ and $S$, respectively. Then we say the source and target charges
are well-separated if $dist(T,S)\ge\frac{1}{2}(diam(T)+diam(S))$.

\subsection{Low-rank characteristics of kernel functions}


We say the kernel  function $\mathcal{K}(\mathbf{r}, \mathbf{r}^{\prime})$
is a generalized asymptotically smooth function \cite{bebendorf2000approximation,banjai2008hierarchical} if
\begin{equation}
\label{eqn:assum}|\partial_{\mathbf{r}}^{\alpha}\partial_{\mathbf{r}^{\prime}%
}^{\beta}\mathcal{K}(\mathbf{r}, \mathbf{r}^{\prime})|\leq c(|\alpha|,
|\beta|)(1+k|\mathbf{r}-\mathbf{r}^{\prime}|)^{|\alpha|+|\beta|}|\mathbf{r}-\mathbf{r}^{\prime}|^{-|\alpha|-|\beta|-\tau},
\end{equation}
where parameters $k,\tau\ge 0$ and $\alpha, \beta$ are $d$-dimensional multi-indices. The constant $C(|\alpha|, |\beta|)$ only depends on $|\alpha|$ and $|\beta|.$
Condition (\ref{eqn:assum}) covers a wide-range of kernel functions. For example, for both 2D and 3D Green's functions of Laplace equation, $k=0$,   but $\tau=0$ and $\tau=1$ for former and latter, respectively. For 3D Green's functions of Helmholtz equation, one has $k$ as wave number and $\tau=1$.

Consider the Taylor expansion of $\mathcal{K}(\mathbf{r}, \mathbf{r}^{\prime
})$ around $\mathbf{r}^{*}$, which is the Chebyshev center of $T$:
$\mathcal{K}(\mathbf{r}, \mathbf{r}^{\prime})=\tilde{\mathcal{K}}(\mathbf{r},
\mathbf{r}^{\prime})+R$ with the polynomial
\begin{equation}
\label{eqn:G-app}\tilde{\mathcal{K}}(\mathbf{r}, \mathbf{r}^{\prime}%
)=\sum_{|\upsilon|=0}^{m-1}\frac{1}{\upsilon!}(\mathbf{r}-\mathbf{r}^{*}%
)^{\upsilon}\frac{\partial^{\upsilon} \mathcal{K}(\mathbf{r}^{*},
\mathbf{r}^{\prime})}{\partial\mathbf{r}^{\upsilon}},%
\end{equation}
and the remainder $R$ satisfies
\begin{equation}
\label{eqn:residue}|R| = |\mathcal{K}(\mathbf{r}, \mathbf{r}^{\prime}%
)-\tilde{\mathcal{K}}(\mathbf{r}, \mathbf{r}^{\prime})|\leq\frac{1}%
{m!}|\mathbf{r}-\mathbf{r}^{*}|^{m}\max_{\zeta\in T, |\gamma|=m}\left|
\frac{\partial^{\gamma}\mathcal{K}(\zeta, \mathbf{r}^{\prime})}{\partial
\zeta^{\gamma}}\right|  .
\end{equation}
If the well-separated target (T) and source (S) charges satisfies
\begin{equation}
\label{eqn:condition}\displaystyle{\frac{\max\{diam(T),
diam(S)\}}{dist(T, S)}}<\eta,%
\end{equation}
for a parameter $0<\eta<1$, then Eq. (\ref{eqn:residue}) has the following
estimate,
\begin{equation}
\label{eqn:R-est}|\mathcal{K}(\mathbf{r}, \mathbf{r}^{\prime})-\tilde
{\mathcal{K}}(\mathbf{r}, \mathbf{r}^{\prime})|\leq c(m)\eta^{m}%
(1+k|{\bf r}-{\bf r}'|)^m(|{\bf r}-{\bf r}'|)|^{-\tau}, \quad
\mathbf{r}\in T, \mathbf{r}^{\prime}\in S.
\end{equation}
Estimate (\ref{eqn:R-est}) implies that for small parameter $k$ , if one
replaces $\mathcal{K}(\mathbf{r}_{i}, \mathbf{r}_{j})$ in Eq.
(\ref{eqn:original}) by $\tilde{\mathcal{K}}(\mathbf{r}_{i}, \mathbf{r}_{j})$
as defined in Eq. (\ref{eqn:G-app}) with error $O(\eta^{m})$, the corresponding interaction matrix
$\tilde{\mathcal{K}}$ is of low rank, i.e.
\begin{equation}
\label{eqn:eta}\mathrm{rank }(\tilde{\mathcal{K}}) = K(m) \ll\min\{M,N\} =
\mathrm{rank }(\mathcal{K}),
\end{equation}
where
\begin{equation}
k(m) = \sum_{p=0}^{m-1}\frac{(d-1+p)!}{(d-1)!p!}.%
\end{equation}
In the case of 2D,  we have $d=2$ and $k(m)= m(m+1)/2$.

This property indicates that we can replace matrix $\mathcal{K}$ by the low-rank matrix $\tilde{\mathcal{K}}$ with enough accuracy for well-separated
target and source points. Thus, the efficiency of computation could be greatly
improved. However, in many situations, there is not easily available explicit formula for
$\tilde{\mathcal{K}}$, as for layered Green's function, thus
the matrix $\tilde{\mathcal{K}}$ cannot not be explicitly computed by Eq.
(\ref{eqn:G-app}). In our approach, we will take advantage of the fact
that $\mathcal{K}$ has redundant information (essential low rank
characteristics) and will use randomized sampling methods to select a small amount of its rows or columns and
approximate the full matrix with the sampled sub-matrices. It should be noted that the low-rank characteristics depend on parameter $k$ and $\tau$ in
the decaying condition of the kernel (\ref{eqn:assum}).

\subsection{Random kernel compression algorithms}

For matrix $\mathbf{A}\in\K^{M\times N}$ with a rank $K\ll\min\{M, N\}$,  the
matrix-vector multiplication with vector $\mathbf{x}\in\K^{N}$ can be
represented as
\begin{equation}
\label{eqn:svd}\mathbf{Ax}=\sum_{i=1}^{K}\sigma_{A}^{i}\mathbf{U}_{A}%
^{(i)}{\mathbf{V}_{A}^{(i)}}^{*}\mathbf{x},
\end{equation}
where $\mathbf{A}=\mathbf{U}_{A}\Sigma_{A}\mathbf{V}^{*}_{A}$, $\mathbf{U}_{A}\in\K^{M\times M}$, $\mathbf{V}_{A}\in\K^{N\times N}$ is the SVD
decomposition of matrix $\mathbf{A}$. Although Eq. (\ref{eqn:svd}) includes
small amount of summation, performing SVD to obtain $\sigma_{A}^{i}$,
$\mathbf{U}_{A}^{i}$ and $\mathbf{V}_{A}^{i}$ is much more expensive than the
direct multiplication. Instead, we approximate the singular values and unitary
matrices by fast Monte Carlo methods. This process is outlined as follows:

Let $\mathbf{C}\in\K^{M\times c}$ be the matrix made of $c$ columns sampled from
matrix $\mathbf{A}$ and denote $\mathbf{C}=\mathbf{U}_{c}\Sigma_{c}%
\mathbf{V}_{c}^{*}$, where $\mathbf{U}_{c}\in\K^{M\times M}$ and $\mathbf{V}_{c}\in\K^{c\times c}$. Further, let $\mathbf{C}_{r}\in\K^{r\times c}$ be the
matrix made of $r$ rows sampled from $\mathbf{C}$ and denote $\mathbf{C}%
_{r}=\mathbf{U}_{r}\Sigma_{r}\mathbf{V}_{r}^{*}$, where $\mathbf{U}_{r}\in\K^{r\times r}$ and $\mathbf{V}_{r}\in\K^{c\times c}$.

\begin{itemize}
\item SVD is performed for the much smaller matrix $\mathbf{C}_{r}$, in which
$r,c\le K\ll\min\{M,N\}$ and independent of $M$ and $N$. Due to rapid decay of
singular values, only first $l$ columns of $\mathbf{V}_{r}$ are needed, where
$l$ is determined by checking $|\sigma_{r}^{l+1}-\sigma_{r}^{l}|<\epsilon
$, where $\epsilon$ is a preset accuracy criteria.\newline

\item We approximate $\mathbf{U}_{c}$ by $\tilde{\mathbf{U}}_{c}\in\K^{M\times
l}$. To obtain $\tilde{\mathbf{U}}_{c}$, we further approximate $\mathbf{U}%
_{c}\Sigma_{c} = \mathbf{CV}_{c} \approx\mathbf{CV}_{r}$, and compute
$\tilde{\mathbf{U}}_{c}$ by $\tilde{\mathbf{U}}_{c}\mathbf{R}=\mathbf{CV}_{r}%
$, where $\mathbf{R}$ is the upper triangle matrix resulting from the QR
decomposition of $\mathbf{CV}_{r}$.

\item Finally we take the approximation $\mathbf{U}_{A}\approx\tilde
{\mathbf{U}}_{A} = \tilde{\mathbf{U}}_{c}$ and $\mathbf{A}\mathbf{x}%
\approx\tilde{\mathbf{U}}_{A} (\tilde{\mathbf{U}}^{*}_{A}\mathbf{A})\mathbf{x}
$. Note that exact computation of $\tilde{\mathbf{U}}^{*}_{A}{\bf A}$ still requires
$O(lMN)$ operations, so the Monte Carlo Basic matrix multiplication algorithm
in \cite{drineas2006fast1} is adopted to reduce the complexity to $O(clN)$.
\end{itemize}

Overall the random kernel compression will have $O(\max\{M,N\})$ operations
and detailed implementations are given in Algorithm \ref{alg:svd}.
\begin{algorithm}[ht!]
		\caption{Random kernel compression}\label{alg:svd}
		\begin{algorithmic}
			\Require matrix ${\bf A}\in\K^{M\times N}$, $c, r,\in\N$, such that $1<c\ll N$ and $1<r\ll M$. A constant $\epsilon > 0$.
			\Ensure  $\sigma_t\in\R^+$, orthonormal vectors ${\bf U}_t\in\K^{M}$, and ${\bf V}_t\in\K^{N}$, for $t=1,2...,l\ll \min{(M, N)}$, such that ${\bf A}\approx\sum_{t=1}^{l}\sigma_t{\bf U}_t{\bf V}_t^*$.
			\State 1. Construct matrix ${\bf C}\in\K^{M\times c}$:
			\For{ t = 1 to c}
				\State (a) pick $i_t\in 1,2,...N$ randomly using uniform distribution;
				\State (b)  set ${\bf C}^{(t)}=\sqrt{N/c}{\bf A}^{(i_t)}$;
			\EndFor
			\State 2. Construct matrix ${\bf C}_r\in\R^{r\times c}$:
			\For{ t = 1 to $r$}
				\State (a) pick $j_t\in 1,2,...M$ randomly using uniform distribution;
				\State (b)  set ${{\bf C}_r}_{(t)}=\sqrt{M/r}{\bf C}_{(j_t)}$;
			\EndFor
			\State 3. Perform SVD on matrix ${\bf C}_r$, i.e., ${\bf C}_r = {\bf U}_r{\bf \Sigma_r}{\bf V}_r^*$ and denote $\sigma({\bf C}_r)$ as the singular values.
			\State 4. Let $l=\min{\{r, c,  \max{\{j, \sigma_j({\bf C}_r)\}}>\epsilon\}}$:
			\For{t =1 to $l$}
				\State (a) $\sigma_t = \sigma_{t}({\bf C_r})$;
				\State (b) ${\bf U}^c_t= {\bf C}{\bf V}_r$
			\EndFor
			\State 5. Perform QR decomposition on $\{{\bf U}^c_t\}$ to obtain the output orthonormal vectors $\{{\bf U}_t\}_{t=1}^{l}$.
			\State 6. Approximate ${\bf V} = {\bf A}^{*}{\bf U}_t$   by the Monte Carlo Basic matrix multiplication algorithm
in \cite{drineas2006fast1}, with $c$ columns from ${\bf A}^*$ and $c$ rows from ${\bf U}_t$;
			\State 7. Perform QR decomposition on ${\bf V}\in\K^{N\times l}$ to obtain the output orthonormal vectors $\{{\bf V}_t\}_{t=1}^{l}$
		\end{algorithmic}
	\end{algorithm}

	Note that in this algorithm, only the small matrix ${\bf C}_r$ is stored in memory and entries of other matrices are calculated on the fly. The idea of Algorithm \ref{alg:svd} is similar to the ``ConstantTimeSVD  algorithm'' in \cite{drineas2006fast2}. However, in that work, the columns ${\bf U}_t$ are directly calculated as ${\bf C}{\bf V}_r/\sigma_{t}({\bf C_r})$ and not orthonormal. More importantly, when ${\bf V}_r$ is not close enough to ${\bf V}_c$ (otherwise efficiency will be compromised), dividing rapidly decaying singular values is not numerically stable. There is no such issues in Algorithm \ref{alg:svd}.

\section{Analysis of the compression algorithm for well-separated sets}\label{sec:analysis}

In this section we investigate the error
\begin{equation}
\|(\mathbf{A}-\Pi\mathbf{A})\mathbf{x}\|\le\|(\mathbf{A}-\Pi\mathbf{A}%
)\|\|\mathbf{x}\|,
\end{equation}
where $\Pi\mathbf{A}$ is the random compression of $\mathbf{A}$ generated from Algorithm
\ref{alg:svd}.
Unfortunately, a full analysis on $\Pi\mathbf{A}$  with general sampling probability is technically complicated. We
would rather consider a theoretically simpler  case about
\begin{equation}\label{eqn:linear}
\mathbf{A}-\tilde{\Pi}\mathbf{A}, \quad \text{with } \tilde{\Pi} = {\bf U}_K{\bf U}^*_K
\end{equation}
where $ {\bf U}_K$ is the matrix  containing only the first $K$ columns of ${\bf U}_c$ in ${\bf C}={\bf U}_c\Sigma_c{\bf V}_c^*$. Matrix ${\bf C}$ here is the same as in  Algorithm \ref{alg:svd}, but could be associated with sampling arbitrary probability.  Construction of the projector $\tilde{\Pi}$ is equivalent to the ``LinearTimeSVD'' algorithm in \cite{drineas2006fast2}. It is computationally inefficient but theoretically simple. We cite the result from \cite{drineas2006fast2}:

\begin{theorem}
[{\small \emph{Theorem 2 and 3 in \cite{drineas2006fast2}}}]\label{thm:bd0}
Suppose
$\mathbf{A}\in\K^{M \times N}$ and $\mathbf{C}\in \K^{M\times c}$ being the
column sampled matrix from \textbf{A} as in Algorithm \ref{alg:svd}. Let $\tilde{\Pi}$ be the  projector
defined in Eq. (\ref{eqn:linear}), then
\begin{equation}
\|{\bf A}-\tilde{\Pi} {\bf A}\|^{2}_{F} \le\|\mathbf{A}-\mathbf{A}_{K}\|^{2}_{F} + 2\sqrt
{K}\|\mathbf{A}\mathbf{A}^{*}-\mathbf{C}\mathbf{C}^{*}\|_{F};
\end{equation}
and
\begin{equation}
\|{\bf A}-\tilde{\Pi} {\bf A}\|^{2}_{2} \le\|\mathbf{A}-\mathbf{A}_{K}\|^{2}_{2} + 2\|\mathbf{A}%
\mathbf{A}^{*}-\mathbf{C}\mathbf{C}^{*}\|_{2};
\end{equation}
where $\mathbf{A}_{K}$ the best $K$-rank approximation of $\mathbf{A}$.
\end{theorem}

Although the theoretical projector $\tilde{\Pi}$ is different from projector $\Pi$,  the one   actually used in our practical computations,  we still can obtain meaningful insights about sampling strategies and accuracy of Algorithm \ref{alg:svd} by examining Theorem \ref{thm:bd0}.

Since $\mathbf{A}$ is already assumed as low-rank, we can only focus on
estimating $\|\mathbf{A}\mathbf{A}^{*}-\mathbf{C}\mathbf{C}^{*}\|_{\xi},
{\xi=2, F}$. We look for a practical sampling probability and derive an error
estimate for well-separated source and target points. For simplicity, we
assume $a = diam(T)=diam(S)$, $\delta= dist (T,S)$, and $a\le\eta\delta$ for
$0<\eta<1$.

\subsection{A nearly optimal uniform sampling of far field kernel matrices}

The low rank property of a matrix $\mathbf{A}$ means most of its columns/rows
are linearly dependent, while each column/rows may contribute significantly
different to the overall matrix property. For example the matrix $A = vv^{T}$ with
$v^{T} = (1,2,3,....N)$ is only of rank one, but the last column/row is the
most important. Therefore, the column/row sampling probability is a critical
factor in minimizing the error $\|\mathbf{A}\mathbf{A}^{*}-\mathbf{C}\mathbf{C}^{*}\|_{\xi
}$.

The optimal probability of (column) sampling to perform a Monte Carlo
matrix-matrix multiplication $\mathbf{AB}$ is proposed in
\cite{drineas2006fast1}.

%

\begin{definition}
[Optimal probability]\label{df:op}
For ${\bf A}\in\K^{M\times N}, {\bf B}\in\K^{N\times P}$, $1\le c\le N$, and $\{p_j\}_{j=1}^N$ such that
\begin{equation}
p_{j} = \frac{|A^{(j)}||B_{(j)}|}{\sum_{j^{\prime}=1}^{N}|A^{(j^{\prime}%
)}||B_{(j^{\prime})}|}, \quad j = 1,2,..., N,
\end{equation}
then for $t=1$ to $c$, pick $i_t\in\{1,2,...,N\}$ with ${\bf Pr}[i_t=j]=p_j, j = 1, 2,...N$ independently and with replacement. Set $C^{(t)}=A^{(i_t)}/\sqrt{cp_{i_t}}$ and $R_{(t)}=B_{(i_t)}/\sqrt{cp_{i_t}}$, the expectation value $\mathbf{E}\left[  \|\mathbf{AB}-\mathbf{CR}\|_{F}\right]$ is minimized.

\end{definition}

Applying this definition to $\mathbf{AA}^{*}$, and using Lemma 4 in
\cite{drineas2006fast1}, it is easy to conclude that if $c$ columns are sampled, the expectation of error
$\|\mathbf{AA}^{*}-\mathbf{CC}^{*}\|^{2}_{F}$ is minimized as
\begin{equation}
\label{eqn:optimal}\mathbf{E}\left[  \|\mathbf{AA}^{*}-\mathbf{CC}^{*}%
\|^{2}_{F}\right]  = \frac{1}{c}\left(  \|\mathbf{A}\|_{F}^{4}-\|\mathbf{AA}%
^{*}\|_{F}^{2}\right).
\end{equation}
However, unless known in advance, utilizing the optimal probability is not
practical because its computation is even be more expensive than the actual
matrix-vector multiplication.

However, a nearly optimal probability is introduced in \cite{drineas2006fast1}, which will be
used for our approach.

\begin{definition}
[Nearly optimal probability]%
For the same conditions in Definition \ref{df:op}, the probability $\{p_j\}$ is called a nearly optimal probability if
\begin{equation}
p_{j} \ge\frac{\beta|A^{(j)}||B_{(j)}|}{\sum_{j^{\prime}=1}^{n}|A^{(j^{\prime
})}||B_{(j^{\prime})}|}, \quad j = 1,2,..., N,
\end{equation}
for some $0<\beta\le1$.
\end{definition}

With the nearly optimal probability, one has
\begin{equation}
\label{eqn:nearoptimal}\mathbf{E}\left[  \|\mathbf{AA}^{*}-\mathbf{CC}%
^{*}\|^{2}_{F}\right]  \le\frac{1}{\beta c}\|\mathbf{A}\|_{F}^{4}-\frac{1}%
{c}\|\mathbf{AA}^{*}\|_{F}^{2}.%
\end{equation}

Now we will present the following result.
\begin{theorem}\label{thm: beta}
For well-separated source and target points, uniform sampling provides a
nearly optimal probability.
\end{theorem}

\begin{proof}
For Eq. (\ref{eqn:optimal}), the optimal probability is
\begin{equation}
p_{j} = \frac{|A^{(j)}|^{2}}{\|A\|_{F}^{2}}.%
\end{equation}
Recall entries of matrix $\mathbf{A}$ are $\mathcal{K}(\mathbf{r}_{i}, \mathbf{r}_{j})q_{i}$ and denote the distance $r_{ij} = |\mathbf{r}_{i}-\mathbf{r}_{j}|$ for simplicity. Then, for the well-separated source and target points, we have the bound
\[
\delta-a\le r_{ij}\le\delta+a,
\]
since $|\mathcal{K}(\mathbf{r}_{i}, \mathbf{r}_{j})|^2$ is monotonically
decreasing with $r_{ij}$, we have
\begin{equation}
p_{j} = \frac{|A^{(j)}|^{2}}{\|A\|_{F}^{2}}\le\frac{|\mathcal{K}(\delta-a)|^{2}\sum_{i=1}^{M}|q_{i}|^{2}}%
{N|\mathcal{K}(\delta+a)|^{2}\sum_{i=1}^{M}|q_{i}|^{2}}.
\end{equation}
Thus, the uniform sampling probability
\begin{equation}
\hat{p}_{j} = \frac{1}{N} > \frac{|\mathcal{K}(\delta+a)|^{2}}{|\mathcal{K}%
(\delta-a)|^{2}} \frac{|A^{(j)}|^{2}}{\|A\|_{F}^{2}},%
\end{equation}
is the the nearly optimal probability with $\beta= \displaystyle{\frac
{|\mathcal{K}(\delta+a)|^{2}}{|\mathcal{K}(\delta-a)|^{2}}}<1$.

\end{proof}

Theorem \ref{thm: beta} indicates that for a fixed number of samples and without additional computational effort (uniform sampling), we can achieve the nearly optimal accuracy as in estimate (\ref{eqn:nearoptimal}) in kernel compression for well-separated source and target points.  However, the bound in Eq.
(\ref{eqn:nearoptimal}) depends on parameter $\beta$, which is an indicator of
how ``well'' the two sets are separated. In case the two sets ``touching''
each other, one has $\eta\to 1$ or $\delta-a\to0$ and hence $\beta\to0$. As a result, the error bound
(\ref{eqn:nearoptimal}) fails.

\begin{figure}[th]
\begin{center}%
\begin{tabular}
[c]{cc}%
\includegraphics[width=0.48\textwidth]{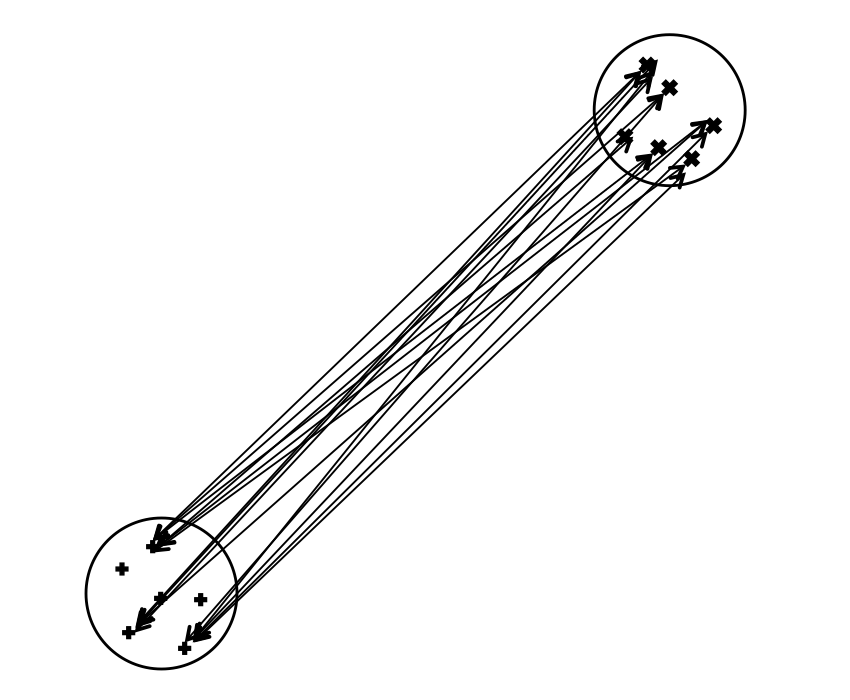} &
\includegraphics[width=0.48\textwidth]{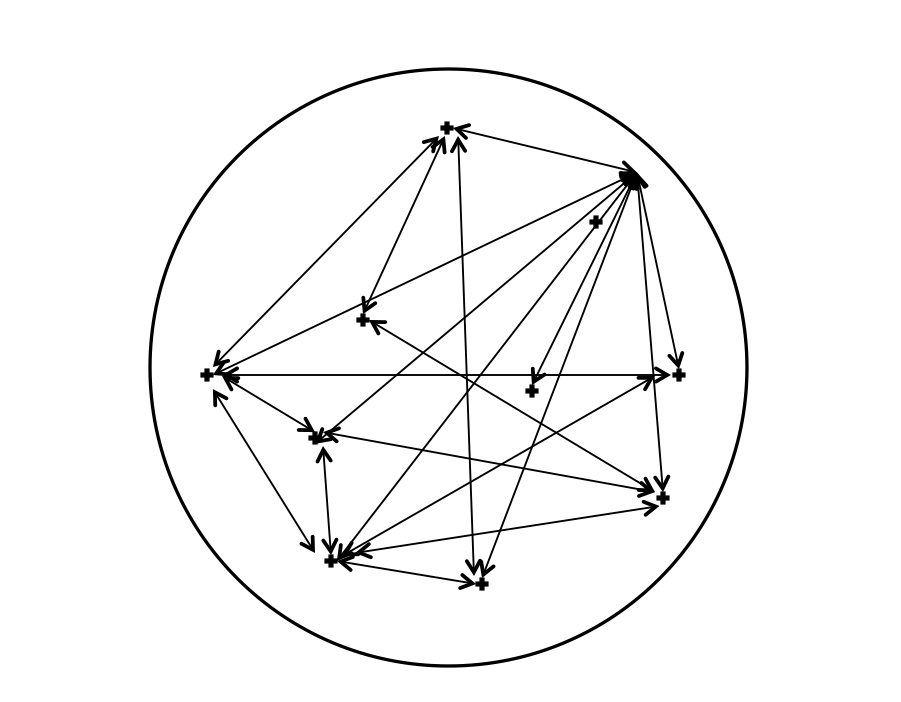}\\
(a) & (b)
\end{tabular}
\end{center}
\caption{(a) well-separated; (b) not well-separated.}%
\label{fig:interactions}%
\end{figure}
The reason why uniform sampling works  can also be illustrated heuristically by Fig.\ref{fig:interactions}: When target and source sets are well-separated and have a small diameter-distance ratio $\eta$ as in Fig. \ref{fig:interactions}(a), if we identify the interaction between target/source points as lines connecting the sources and targets, it can be seen that all interactions are ``similar'', in terms of direction and magnitude, to each other. Then in this case, columns or row vectors in the matrix have fairly the same contribution, so uniform sampling will have a small error. On the other hand, if target and source points belong to the same set as in Fig. \ref{fig:interactions}(b), the interactions are rather different from each other. Even the matrix maybe low rank, uniform sampling will yield in uncontrollable error due to the complexity of angles of the "interaction" lines.

In the next subsection we provide a further quantitative analysis of Eqs. (\ref{eqn:optimal}) or (\ref{eqn:nearoptimal}) upon the separation.  We show that even the  error  (\ref{eqn:optimal}) or (\ref{eqn:nearoptimal}) depends on not only sample number $c$, but also the diameter-distance ratio $\eta$.

\subsection{Error analysis on target-source separation}

Error estimates (\ref{eqn:optimal}) and (\ref{eqn:nearoptimal}) provided in
\cite{drineas2006fast2} are for a general matrix $\mathbf{A}%
$, stating that the error is small enough if the samples are large enough. But in practice, the sample number $c$ can not be too large due to efficiency requirement.
Here, we give a finer estimate for $\mathbf{A}$ corresponding to
well-separated target and source points. We show that for some type of kernels, if the target and source
sets are far way enough, the error is small enough regardless of the sample
numbers. Since we are more interested in the dependence of error bound on
diameter-distance ration, thus consider Eq.
(\ref{eqn:optimal}) for simplicity.

\begin{theorem}\label{thm:bd1}
Suppose $\mathbf{A}\in\K^{M \times N}$ is the matrix generated from the kernel function $\mathcal{K}({\bf r}, {\bf r}')$ satisfying condition (\ref{eqn:assum}), and $\mathbf{C}\in\K^{M\times c}$ being the column sampled matrix from \textbf{A}. Let $\tilde\Pi$ be
the orthogonal projector defined in Eq. (\ref{eqn:linear}), then
\begin{equation}\label{eqn:bd1}
\|{\bf A}-\tilde{\Pi} {\bf A}\|^{2}_{F} \le\|\mathbf{A}-\mathbf{A}_{K}\|^{2}_{F}+2\sqrt{K}\|\mathbf{A}\mathbf{A}^{*}-\mathbf{C}%
\mathbf{C}^{*}\|_{F};
\end{equation}
or
\begin{equation}\label{eqn:bd2}
\|{\bf A}-\tilde{\Pi} {\bf A}\|^{2}_{2} \le\|\mathbf{A}-\mathbf{A}_{K}\|^{2}_{2}+2\|\mathbf{A}\mathbf{A}^{*}-\mathbf{C}\mathbf{C}%
^{*}\|_{2};
\end{equation}
and
\begin{equation}\label{eqn:bd3}
\mathbf{E}\left[  \|\mathbf{AA}^{*}-\mathbf{CC}^{*}\|_{F}\right]  \le
C(\delta, a, \tau, k, M, N)\frac{1}{\sqrt{c}}\frac{2\eta}{2-\eta} \|\mathbf{A}\|_{F}^{2},
\end{equation}
where $\delta$ is the distance, $a$ is the diameter, and $0<\eta<1$ is the diameter-distance ratio of target and source sets.
\end{theorem}

\begin{proof}
Equations (\ref{eqn:bd1}) and (\ref{eqn:bd2}) have been stated in Theorem \ref{thm:bd0}. To prove (\ref{eqn:bd3}),
write the $i$-th row of $\mathbf{A}$ as 
\begin{equation}
\label{eqn:col}A_{(i)}=q_{i}(\mathcal{K}(\mathbf{r}_{i}, \mathbf{r}_{1}),
\mathcal{K}(\mathbf{r}_{i}, \mathbf{r}_{2}),..., \mathcal{K}(\mathbf{r}_{i},
\mathbf{r}_{N})),\quad i = 1,2,...M.
\end{equation}
Then, the difference between $q_{i^{\prime}}A_{(i)}$ and $q_{i}A_{(i^{\prime})}$ is,
approximated to the first order,
\begin{equation}
\Delta_{ii^{\prime}} = q_{i}q_{i^{\prime}}(\Delta_{ii^{\prime}}^{(1)},
\Delta_{ii^{\prime}}^{(2)},...\Delta_{ii^{\prime}}^{(N)}),
\end{equation}
where
\begin{equation}
\Delta_{ii^{\prime}}^{(j)} = (\mathbf{r}_{i}-\mathbf{r}_{i^{\prime}}%
)\cdot\frac{\partial}{\partial\mathbf{r}}\left.  \mathcal{K}(\mathbf{r},
\mathbf{r}_{j})\right|  _{\mathbf{r}=\mathbf{r}^{*}_{T}}.
\end{equation}
Recall assumption in Eq. (\ref{eqn:assum}) and condition (\ref{eqn:condition}%
), we have estimate
\begin{eqnarray}\label{eqn:deltaest}
|\Delta_{ii^{\prime}}^{(j)} |&\le&\frac{a}{\delta-\frac
{a}{2}}(1+k(\delta+a))\left(\delta-\frac{a}{2}\right)^{-\tau}
\le \frac{2\eta}{2-\eta}(1+2k\delta)\left(\frac{\delta}{2}\right)^{-\tau}.
\end{eqnarray}

Rewrite Eq. (\ref{eqn:optimal}) as
\begin{align}
\label{eqn:exp} &  \mathbf{E}\left[  \|\mathbf{AA}^{*}-\mathbf{CC}^{*}%
\|^{2}_{F}\right]  = \frac{1}{c}\left(  \|\mathbf{A}\|_{F}^{4}-\|\mathbf{AA}%
^{*}\|_{F}^{2}\right) \nonumber\\
&  =\frac{1}{c}\left[  \left(  \sum_{i=1}^{M}|A_{(i)}|^{2}\right)  ^{2}%
-\sum_{i=1}^{M}\sum_{i^{\prime}=1}^{M}\langle A_{(i)}, A_{(i^{\prime})}%
\rangle^{2}\right] \nonumber\\
&  = \frac{1}{c}\sum_{i=1}^{M}\sum_{i^{\prime}=1}^{M}\frac{1}{q^{2}_{i}%
q^{2}_{i^{\prime}}}\left[  \langle q_{i^{\prime}}A_{(i)}, q_{i^{\prime}%
}A_{(i)}\rangle\langle q_{i}A_{(i^{\prime})}, q_{i}A_{(i^{\prime})}%
\rangle-\langle q_{i^{\prime}}A_{(i)}, q_{i}A_{(i^{\prime})}\rangle
^{2}\right],
\end{align}
where $\langle,\rangle$ represent inner product.

Since $q_{i}A_{(i^{\prime})} = q_{i^{\prime}}A_{(i)} + \Delta_{ii^{\prime}} $,  using the linearity of inner product,
we have
\begin{align}
\label{eqn:diff} &  \langle q_{i^{\prime}}A_{(i)}, q_{i^{\prime}}%
A_{(i)}\rangle\langle q_{i}A_{(i^{\prime})}, q_{i}A_{(i^{\prime})}%
\rangle-\langle q_{i^{\prime}}A_{(i)}, q_{i}A_{(i^{\prime})}\rangle
^{2}\nonumber\\
&  = \langle q_{i^{\prime}}A_{(i)}, q_{i^{\prime}}A_{(i)}\rangle\langle
\Delta_{ii^{\prime}}, \Delta_{ii^{\prime}}\rangle- \langle q_{i^{\prime}%
}A_{(i)}, \Delta_{ii^{\prime}}\rangle^{2}\nonumber\\
&  \le\langle q_{i^{\prime}}A_{(i)}, q_{i^{\prime}}A_{(i)}\rangle\langle
\Delta_{ii^{\prime}}, \Delta_{ii^{\prime}}\rangle.
\end{align}
Plugging (\ref{eqn:diff}) in (\ref{eqn:exp}) and using estimate (\ref{eqn:deltaest}), we arrive at
\begin{align}
\label{eqn:exp2}\mathbf{E}\left[  \|\mathbf{AA}^{*}-\mathbf{CC}^{*}\|^{2}%
_{F}\right]   &  \le\frac{1}{c}\sum_{i=1}^{M}\sum_{i^{\prime}=1}^{M}\langle
A_{(i)}, A_{(i)}\rangle\langle\Delta_{ii^{\prime}} \Delta_{ii^{\prime}}%
\rangle\nonumber\\
&  \le\frac{M\|Q\|_{2}^{2}}{c}\left(  \frac{2\eta}{2-\eta}\right)  ^{2}(1+2k\delta)^2\left(\frac{\delta}{2}\right)^{-2\tau} \|\mathbf{A}\|_{F}^{2},%
\end{align}
where $Q = (q_{1},q_{2},...q_{M})$ is the density of target points. By Jensen's inequality
		\begin{eqnarray}\nonumber
			{\bf E}\left[\|{\bf AA}^*-{\bf CC}^*\|_F\right] &\le& \frac{\sqrt{M}\|Q\|_2}{\sqrt{c}\|{\bf A}\|_F}(1+2k\delta)\left(\frac{\delta}{2}\right)^{-\tau}\frac{2\eta}{2-\eta}	\|{\bf A}\|_F^2.\\\nonumber
				&\le & \frac{\sqrt{M}\|Q\|_2(1+2k\delta)\left(\frac{\delta}{2}\right)^{-\tau}}{\sqrt{N}\|Q\|_2\left|\mathcal{K}(\delta+a)\right|}\frac{1}{\sqrt{c}}\frac{2\eta}{2-\eta}	 \|{\bf A}\|_F^2\\
				&=&C(\delta, a, \tau,  k, M, N)\frac{1}{\sqrt{c}}\frac{2\eta}{2-\eta}	\|{\bf A}\|_F^2,
		\end{eqnarray}		
		 where $C(\delta, a, \tau, k, M, N)=\displaystyle{\left(\frac{\delta}{2}\right)^{-\tau}\frac{(1+2k\delta)}{\left|\mathcal{K}(\delta+a)\right|}\frac{\sqrt{M}}{\sqrt{N}}}$. 
		\end{proof}
		
		This result indicates that the error of kernel compression also depends on the diameter-distance ratio of the well-separated target and source points.
		It is important to point out that  in the constant $C(\delta, a, \tau, k, M, N)$, parameter $k$ is critical to the compression error. A typical example is the Green's function for Helmholtz equation, for which the high frequency problem is always a challenge for any kernel compression algorithm.
		

\section{Hierarchical matrix ($\mathscr{H}$-matrix) structure for general data sets}
	For general cases when target and source are not well-separated, in fact typically they belong to the same set, we logically partition the whole matrix into blocks, each of which will have low-rank approximation as being associated with a far-field interaction sub-matrix, thus the kernel compression algorithm applies hierarchically at different scales. We call the resulting method ``hierarchical random compression method (HRCM)".
\subsection{Review of $\mathscr{H}$-matrix}

\begin{definition}
[Hierarchical matrices ($\mathscr{H}$-matrix)]\label{def:hmatrix} Let $I$ be a
finite index set and $P_{2}$ be a (disjoint) block partitioning (tensor or
non-tensor) of $I\times I$ and $K\in\N$. The underlying field of the vector
space of matrices is $\in\{, \}$. The set of $\mathscr{H}$-matrix induced by
$P_{2}$ is
\begin{equation}
\mathscr{M}_{\mathscr{H},K} :=\{\mathbf{M}\in^{I\times I}: \text{each block }
\mathbf{M}^{b}, b\in P_{2}, \text{satisfies } rank(\mathbf{M}^{b})\le K\}.
\end{equation}

\end{definition}

\textbf{Remarks:} (1) The index set $I$ can be the physical coordinates of
target/source points; we denote a matrix $\mathbf{A} $ as R-$K$ matrix if
$rank(\mathbf{A})\le K$; (2) A specific $\mathscr{H}$-matrix is defined
through \ref{def:hmatrix} recursively. Full definition, description, and
construction of $\mathscr{H}$-matrices are given in detail in
\cite{hackbusch1999sparse,hackbusch2000sparse}. Two simple examples are given
as follows; (3) Since $\mathscr{H}$-matrix is recursive, we always assume
$\mathbf{A}\in\K^{N\times N}$ and $N = 2^{p}$ for the following context.

\noindent\textbf{Example One:} $\mathbf{A}\in\mathscr{M}_{\mathscr{H}, K}$ if
either $2^p=K$ or it has the structure
\[
\mathbf{A}=%
\begin{bmatrix}
\mathbf{A}_{11} & \mathbf{A}_{12}\\
\mathbf{A}_{21} & \mathbf{A}_{22}%
\end{bmatrix}
\text{ with } \mathbf{A}_{11}, \mathbf{A}_{22} \in\mathscr{M}_{\mathscr{H},K}
\text{ and R-}K \text{ matrices } \mathbf{A}_{12}, \mathbf{A}_{21}.
\]
This is the simplest $\mathscr{H}$-matrix. Such a matrix with three levels of division is
visualized in the left of Fig. \ref{fig:hA}. All blocks are R-$K$ matrices, except the smallest ones.

The next example is more complicated and it includes another recursive
concept of neighborhood matrix. Construction of this $\mathscr{H}$-matrix includes
three steps.

\noindent\textbf{Example Two:} (i) Neighborhood matrix
$\mathscr{M}_{\mathcal{N},K}$: if either $2^p=K$ or it has the structure
\[
\mathbf{A}=%
\begin{bmatrix}
\mathbf{A}_{11} & \mathbf{A}_{12}\\
\mathbf{A}_{21} & \mathbf{A}_{21}%
\end{bmatrix}
\text{ with } \mathbf{A}_{21} \in\mathscr{M}_{\mathcal{N},k} \text{ and R-}K
\text{ matrices } \mathbf{A}_{11}, \mathbf{A}_{12}, \mathbf{A}_{22}.
\]

(ii) $\mathbf{A}\in\mathscr{M}_{\mathcal{N}^{*},k}$ if $\mathbf{A}^{*}%
\in\mathscr{M}_{\mathcal{N},K}$.

(iii) $\mathbf{A}\in\mathscr{M}_{\mathscr{H},K}$ if either $2^p=K$ or it has the
structure
\[
\mathbf{A}=%
\begin{bmatrix}
\mathbf{A}_{11} & \mathbf{A}_{12}\\
\mathbf{A}_{21} & \mathbf{A}_{21}%
\end{bmatrix}
\text{ with } \mathbf{A}_{11}, \mathbf{A}_{22} \in\mathscr{M}_{\mathscr{H},K},
\mathbf{A}_{12}\in\mathscr{M}_{\mathcal{N},K}, \text{ and }\mathbf{A}_{21}%
\in\mathscr{M}_{\mathcal{N}^{*},K}.
\]
\begin{figure}[th]
\begin{center}
\includegraphics[width=0.85\textwidth]{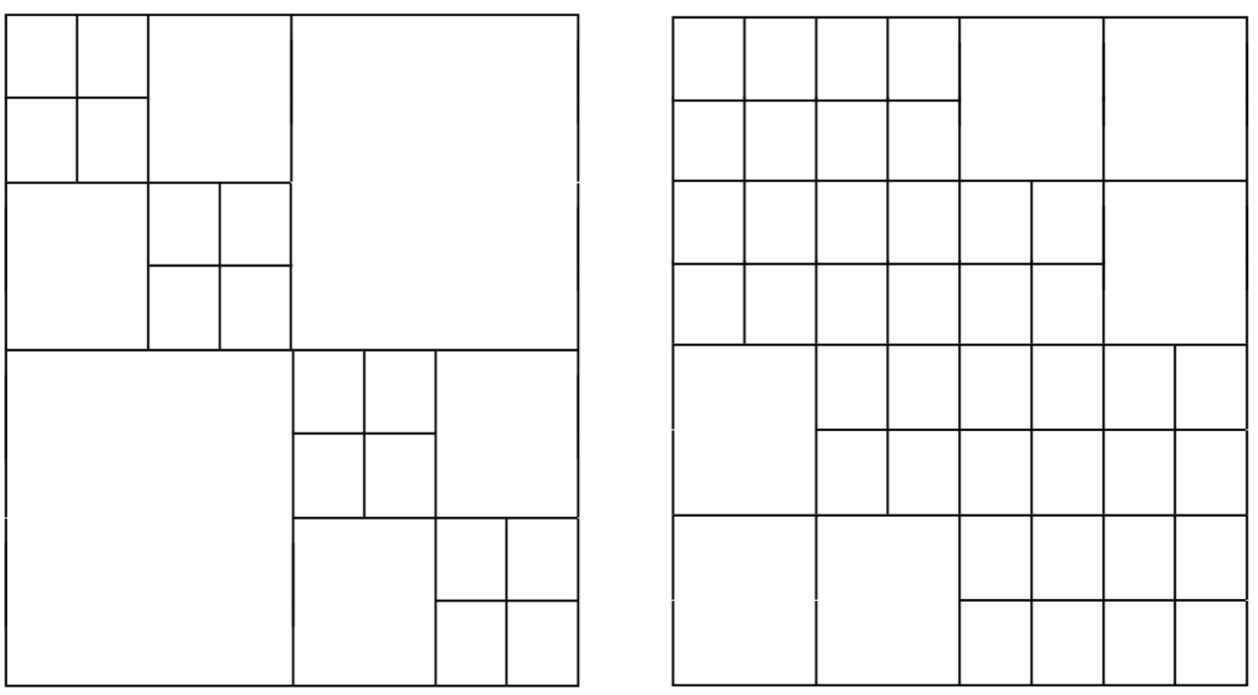}
\end{center}
\caption{Two examples of $\mathscr{H}$-matrix. Left: example one; right:
example two.}%
\label{fig:hA}%
\end{figure}

Visualization of the second example with three levels of division is given in the right subfigure
of Fig. \ref{fig:hA}. Similarly, only the larger blocks are R-$K$ matrices.

With such decomposition,  matrix-vector product will be performed as
\begin{equation}\label{eqn:splitting}
	{\bf A}{\bf x} = {\bf A}_{11}{\bf x}_1 + {\bf A}_{12}{\bf x}_2 +  {\bf A}_{21}{\bf x}_1 + {\bf A}_{22}{\bf x}_2.
\end{equation}
where ${\bf x}^T = ({\bf x}^T_1, {\bf x}^T_2)$. As ${\bf A}$ in Example One, random compression can be immediately applied to ${\bf A}_{12}{\bf x}_2 $ and ${\bf A}_{21}{\bf x}_1$, while recursive division and random compression need to be implemented on ${\bf A}_{11}{\bf x}_1 $ and ${\bf A}_{22}{\bf x}_2$, until a preset minimum block is reached where direct matrix-vector multiplication is used.

But  in Example Two, none of the four terms in Eq. (\ref{eqn:splitting}) is R-$K$ matrix ready for compression. Each ${\bf A}_{ij}$ needs to further divided and investigated. For instance, ${\bf A}_{12}\in\mathscr{M}_{\mathcal{N},K}$ by definition (iii), then by (i), ${\bf A}_{12_{11}}$, ${\bf A}_{12_{12}}$, and ${\bf A}_{12_{22}}$ are R-$K$ matrices and the compression algorithm can be applied. In contrast, ${\bf A}_{12_{21}}$ needs to be further divided.

Matrices in Example One and Two can be understood as interactions of
target/source points along a 1D geometry. Example One indicates that
interactions of all subset of $I$ can be approximated as low-rank compressions
except self-interactions of the subsets, which requires further division.
While the blocks in Example Two requires further division for both
self-interacting and immediate neighboring subsets of $I$.

Structures of $\mathscr{H}$-matrices for target/source points in  high-dimensional geometry are much more complicated. It is difficult to partition them into blocks as shown in Fig \ref{fig:hA}. Instead, we construct the $\mathscr{H}$-matrix logically through a partition tree
of $I$ and the concept of admissible clusters.

\subsection{Tree structure of $\mathscr{H}$-matrix for two dimensional data}

We illustrate algorithms in two-dimensional (2D) case. Here the dimension refers to the geometry where the target and source points are located instead of the dimension of the kernel function.  For
simplicity, let $\Omega=[0, L]\times[0,L]$ and consider a regular grid
\begin{equation}
I = \{(i,j): 1\le i, j \le N_1\}, \quad N_1 = 2^{p}.
\end{equation}
Each index $(i,j)\in I$ is associated with the square
\begin{equation}
X_{ij}: \{(x,y):(i-1)h\le x \le ih, (j-1)h\le y\le jh\}, \quad h = L/N_1,
\end{equation}
in which a certain amount of target/source points are assigned. The
partitioning $T(I)$ of $I$ uses a quadtree, with children (or leaves):
\begin{equation}
t^{l}_{\alpha, \beta} := \{(i,j): 2^{p-l}\alpha+1\le i \le2^{p-l}(\alpha+1),
2^{p-l}\beta+1\le j \le2^{p-l}(\beta+ 1)\},
\end{equation}
with $\alpha, \beta\in\{0,1,..., 2^{l}-1\}$ belong to level $l\in
\{0,1,...,p\}$. We consider a target quadtree $I_t$ and a source quadtree $I_s$,  which could be same or different.
Then  blocks of interaction matrix are defined as
block $b=(t_{1}, t_{2})\in T(I_t\times I_s)$, where $t_{1}, t_2\in I_t\times I_s$  belong to the the same level
$l$. Then follow Eq. (\ref{eqn:diam})-(\ref{eqn:dist}) we define the diameters
and distances of  $t_{1}$ and $t_{2}$, and the
admissibility condition
\begin{equation}
\max\{diam(t_{1}), diam(t_{2})\}\le\eta dist(t_{1}, t_{2}), \quad0<\eta<1,
\end{equation}
for the block $b=(t_{1}, t_{2})$. A block $b$ is called admissible or an
admissible cluster if either $b$ is a leaf or the admissibility condition
holds.  If $b$ is admissible, no matter how many points are in $t_1$ and $t_2$, the
block matrix $\mathbf{M}^{b}$ has rank up to $K$, thus low rank approximation
algorithms are used. Otherwise, both $t_{1}$ and $t_{2}$ will be further
partitioned into children until leaves. And the above process is implemented,
recursively.


\begin{figure}[th]
\begin{center}%
\begin{tabular}
[c]{cc}%
\includegraphics[width=0.35\textwidth]{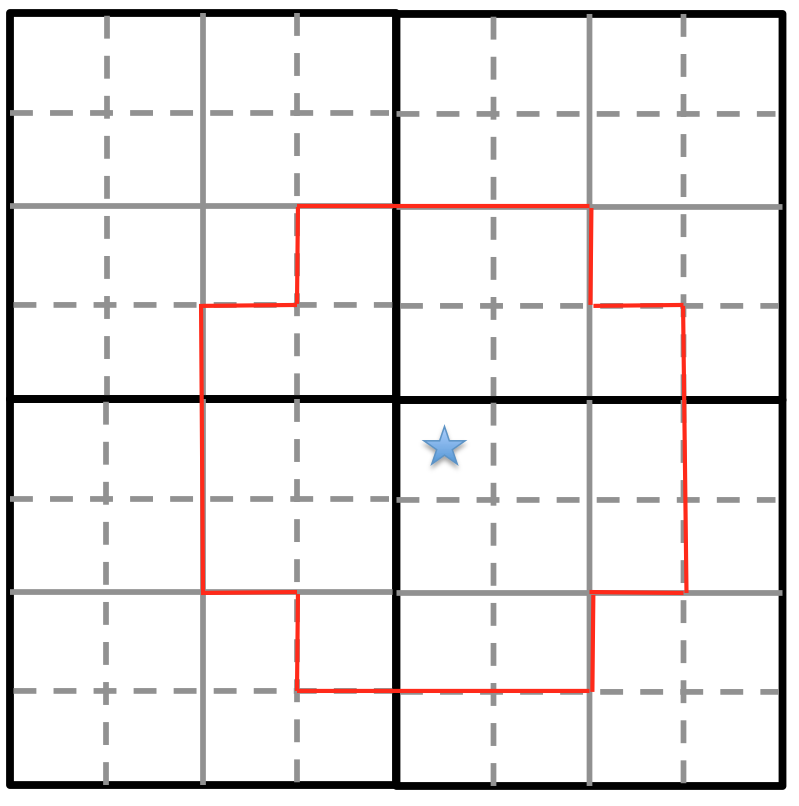} &
\includegraphics[width=0.65\textwidth]{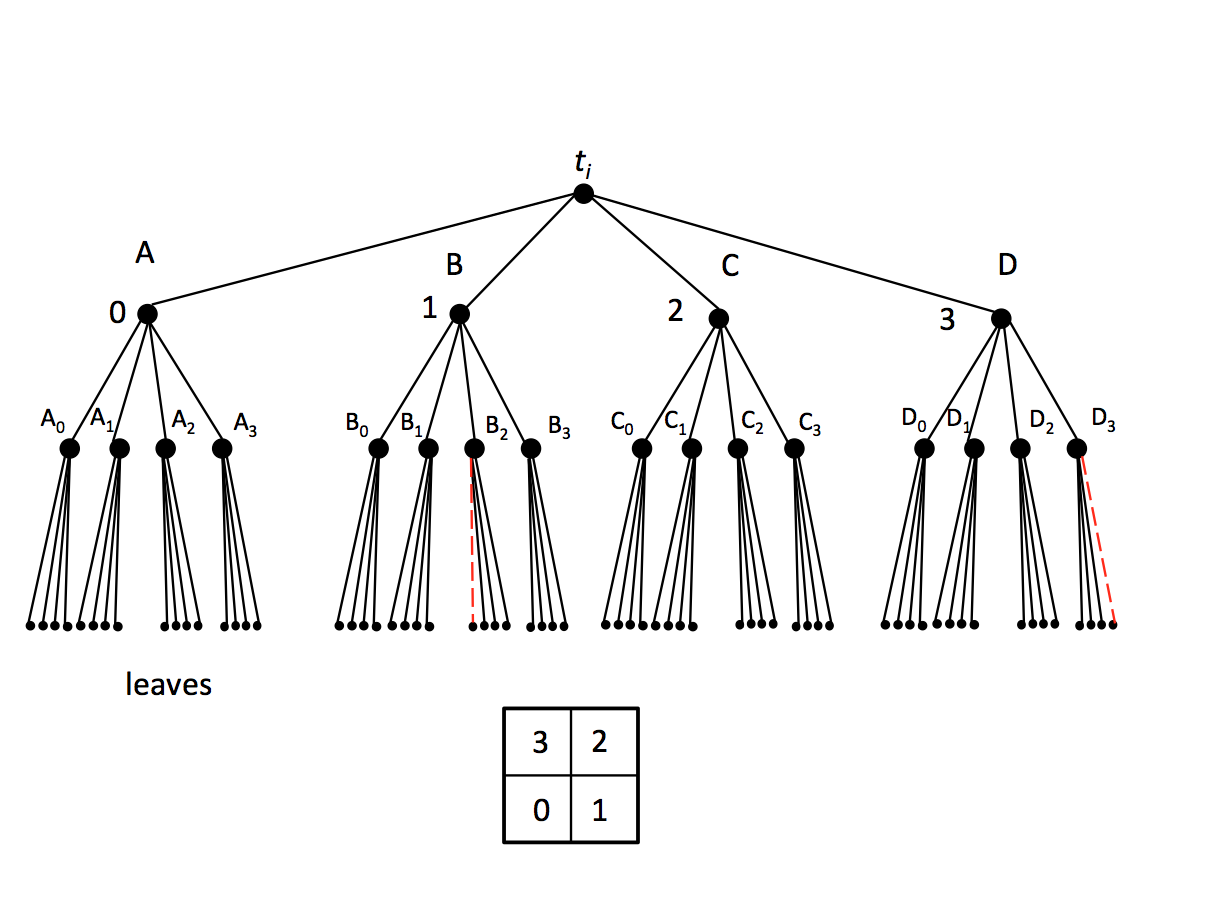}\\
(a) & (b)
\end{tabular}
\end{center}
\caption{Illustration of the index set $I$: (a) admissible cluster for the
starred square; (b) quadtree structure. }%
\label{fig:hmatrix}%
\end{figure}

Figure \ref{fig:hmatrix}(a) shows the index set $I$, where black
solid, gray solid and gray lines are for partition at level $l=1,2,3$,
respectively.  For different values of $\eta$, admissible clusters are
different for a given child. For the square marked with star in Fig
\ref{fig:hmatrix}(a), if $\eta=\sqrt{2}/2$, the non-admissible clusters are itself and the eight
immediate surrounding squares. While for $\eta= 1/2$, any
squares within the the red lines are non-admissible.

 Figure \ref{fig:hmatrix}(b) displays the quadtree that divides each square. Four children of each branch are labeled as $0,1,2,3$ and ordered
counter-clock wisely. If there are $N=4^p$ target (source) points, the depth of the target (source) tree is $p-p_0$, where $p_0$ is the number of points in each leaf, or direct multiplication is performed when matrix size is down to $4^{p_0}\times 4^{p_0}$.

\begin{figure}[th]
\begin{center}
\includegraphics[width=0.75\textwidth]{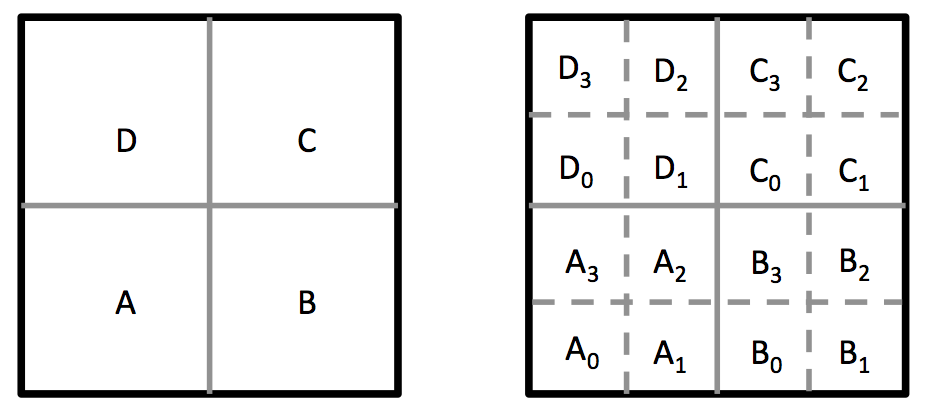}
\end{center}
\caption{Illustration of the partitioning of admissible clusters}%
\label{fig:partition}%
\end{figure}

Figure \ref{fig:partition} illustrates the admissible and non-admissible clusters.
Initially the 2D set is partitioned into four subdomain $A, B, C$ and $D$. Any
two of the subdomains are non-admissible for $\eta=\sqrt{2}/2$. Then each of them
are further divided into four children domain, as label on the right of Fig
\ref{fig:partition}, among which the interactions are examined. For example,
the interactions of $A$ and $B$ can be viewed as the sum of interactions of
$A_{i}$ and $B_{j}$, $i,j=0,1,2,3$. If we take $\eta=1/2$, only $A_{3}$ and
$B_{1}$, and $A_{0}$ and $B_{2}$ are admissible clusters. But if $\eta=
\sqrt{2}/2$, only $A_{2}$ and $B_{0}$, and $A_{1}$ and $B_{3}$ are
non-admissible pairs. For both values of $\eta$, only $A_{2}$ and $C_{0}$ are
non-admissible clusters in the interactions of $A$ and $C$.
Self-interactions, such as interaction between $A$ and $A$, can be viewed as
the same process of interactions among $A, B, C$ and $D$, but for $A_{0},
A_{1}, A_{2}$, and $A_{3}$.

Direct and low rank approximation of matrix-vector multiplications are
implemented on the quadtree. To perform the algorithms, all the index in $I$ is ordered as $i=\{0,1,...N-1\}$ with $N=4^p$. Note that there are totally $4^{p-l}$ points in each child/leaf at level $l$. Matrix column (row) sampling is achieved through sampling of children from level $l+1$ to level $p$. For example, cluster $b=(B_2, D_3)$ in Fig. \ref{fig:hmatrix} (b) is admissible and assume  it is at level $l$. If we generate $s_i\in\{0,1,2,3\}, i=l+1,..p$ randomly and choose only the $s_i$-th child of  $B_2$ (red dash line) at $i$-th level, then one row sampling of  the corresponding (block) kernel matrix is accomplished assuming $B_2\in I_t$. Similarly column sampling is the child-picking process on $D_3\in I_s$.

The full HRCM algorithms are summarized as in the following two algorithms:
 \begin{algorithm}
		\caption{HRCM: Direct product on source and target quadtrees}
		subroutine name: DirectProduct($*$target, $*$source, int level)
		\begin{algorithmic}
			\Require Root pointers of source and target quadtree information for ${\bf A}$ and ${\bf x}$. Current level $l$ and maximum level $p$.
			\Ensure  Product ${\bf y} = {\bf A}{\bf x}$, where ${\bf y}$ is stored in the target quadtree.
			\If{level == maxlevel}
				\State perform and scalar product, and return
			\Else
				\For{j = 0; j $<$ 4; j++}
					\For{i = 0; i $<$ 4; i++}
						\State DirectProduct(target-$>$child[j], source-$>$child[i], level +1)
					\EndFor
				\EndFor
			\EndIf
		\end{algorithmic}
	\end{algorithm}

\begin{algorithm}
		\caption{HRCM: Low-rank product on source and target quadtrees}
		subroutine name: LowRankProduct($*$target, $*$source, int level)
		\begin{algorithmic}
			\Require Root pointers of source and target quadtree information for ${\bf A}$ and ${\bf x}$. Current level $l$ and maximum level $p$.
			\Ensure   ${\bf y} \approx\sum_{t=1}^{l}\sigma_t{\bf U}_t{\bf V}_t^*{\bf x}$, where ${\bf y}$ is stored in the target quadtree.
			
			\State 1. Column sampling: On the source tree, pick the ``random path'' from level $l$ to maxlevel $p$ by only randomly choosing one child from each level;
			\State 2. Row sampling: On the target tree, pick the ``random path'' from level $l$ to maxlevel $p$ by only randomly choosing one child from each level;
			\State 3. Extract matrix entries from the source and target tree by the column/row sampling; perform Algorithm \ref{alg:svd}
			\State 4. Instore ${\bf y}$ into the target tree with root $*$target.
		\end{algorithmic}
	\end{algorithm}

\begin{algorithm}
		\caption{HRCM: $\mathscr{H}$-matrix product on source and target quadtrees}
		subroutine name: HmatrixProduct($*$target, $*$source, int level)
		\begin{algorithmic}
			\Require Root pointers of source and target quadtree  for ${\bf A}$ and ${\bf x}$. Current level $l$ and maximum level $p$.
			\Ensure   $\tilde{\bf y}\approx {\bf y} = {\bf A}{\bf x}$, where $\tilde{\bf y}$ is stored in the target quadtree.
			
			\If{matrix small enough}
				\State DirectProduct($*$target, $*$source, level)
			\Else
				\If{clusters rooted from $*$target, $*$source are admissible}
					\State{LowRankProduct($*$target, $*$source, level)}
				\Else
					\For{j = 0; j $<$ 4; j++}
						\For{i = 0; i $<$ 4; i++}
							\State HmatrixProduct(target-$>$child[j], source-$>$child[i], level +1)
						\EndFor
					\EndFor
				\EndIf
			\EndIf
		\end{algorithmic}
	\end{algorithm}

\subsection{Efficiency analysis}
	
	It is easy to perform efficiency analysis of the hierarchical kernel compression method  by constructing an interaction pattern tree. With $\eta = \sqrt{2}/2$, all the non-admissible clusters can be classified into three interaction patterns: the self-interaction (S), edge-contact interaction (E), and vertex-contact interaction (V), as shown in Fig. \ref{fig:format}. Assume it is currently level $l$ and those target and source boxes need to be further divided into four children in level $l+1$. In level $l+1$, those children form 16 interactions. It is easy to check that from level $l$ to level $l+1$, as displayed in Fig \ref{fig:format},  S-interaction forms 4 S-, 8 E- and 4 V-interactions at level $l+1$. On the other hand, E-interaction forms 2 E-interaction, 2 V-interactions and 12 admissible clusters for which  low-rank approximation (LR) applies.  Additionally, V-interaction forms 1 V-interactions and 15 LR approximations.
	 \begin{figure}[th]
		\begin{center}
		\includegraphics[width=0.85\textwidth]{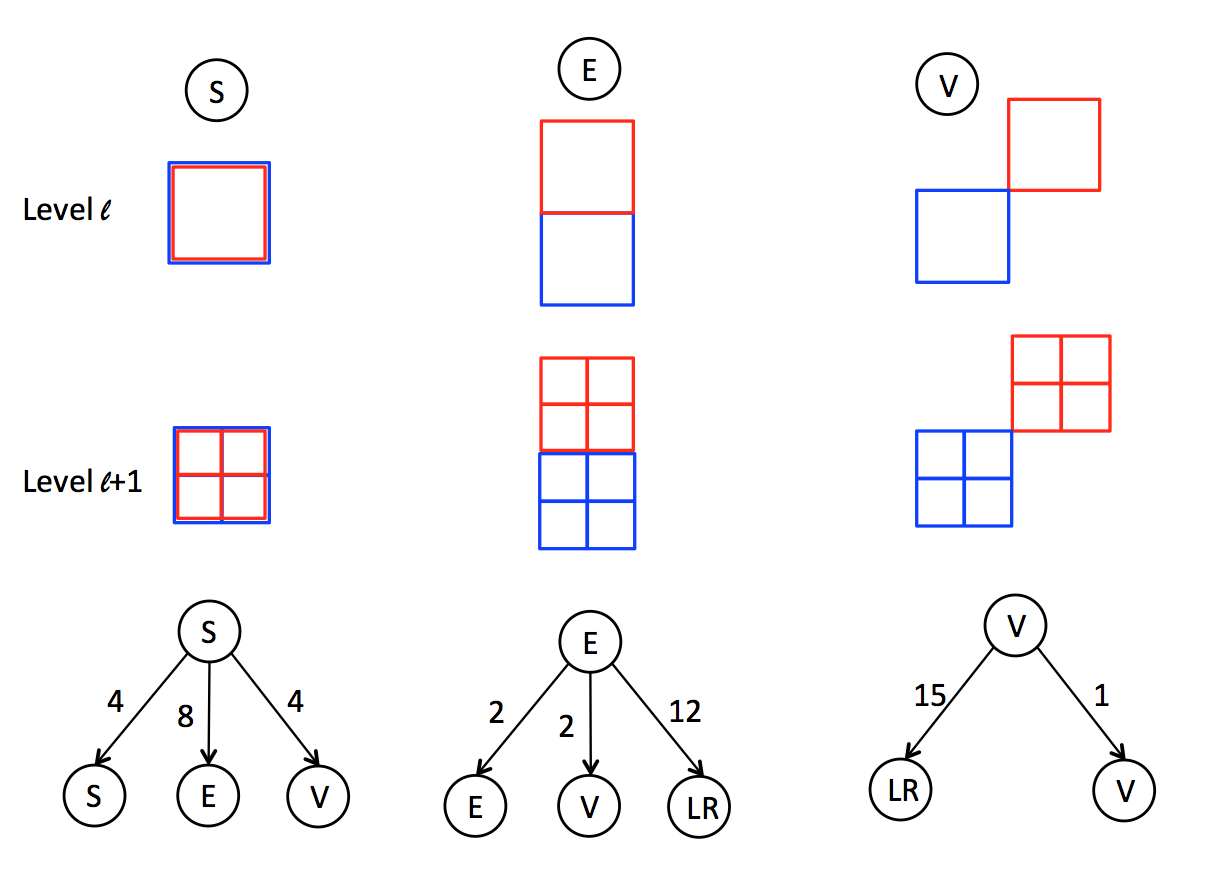}
		 \caption{Evolution of non-admissible clusters. S: self-interaction clusters; E: clusters touch by edge; V: clusters touch by vertex; LR: admissible clusters with low-rank approximation.}%
		\label{fig:format}%
		\end{center}
	\end{figure}
	
	\underline{Complexity of all direct calculations.} We assume the direct computation is implemented when the matrix scale is down to $4^{p_0}\times 4^{p_0}$. So we start from S-interaction as the root at level $l$ and just need to count how many E- and V-interactions  at level $p-p_0-1$ are generated. The  number is
	\begin{equation}
		8\cdot 2^{p-p_0-1-l-1}+ 8\cdot 2\cdot 1^{p-p_0-1-l-2} + 4\cdot 1^{p-p_0-1-l-2}.
	\end{equation}
	Since there are $4^l$ S-interactions at level $l$, the total complexity for performing direct computation is
	\begin{equation}\label{eqn:SD}
		16\cdot O((4^{p_0})^2)\sum_{l=0}^{p-p_0-1}4^l(8\cdot 2^{p-p_0-l-2} + 8\cdot 2+ 4)=O(4^p)=O(N).
	\end{equation}

	\underline{Complexity of all low-rank compressions}. Again we start from S as the root at level $l$ and  then the resulting  E and V start to generate LR at the ($l+2$)-th level and continue to the last level. Recall at level $l$ the complexity of performing LR is $O(4^{p-l})$, so the complexity of performing LR starting from S at level $l$ is
	\begin{eqnarray}\nonumber
		&&\underbrace{(8\cdot 12+4\cdot 15)O(4^{p-l-2})}_\text{from E and V at l+2 level}+\underbrace{\sum_{l'=l+3}^{p-p_0-1}8\cdot 2^{l'-l-2}\cdot 12\cdot O(4^{p-l'})}_\text{from all the E to the bottom}\\\label{eqn:SLR}
		&+&\underbrace{\sum_{l'=l+3}^{p-p_0-1}8\cdot 2\cdot 1^{l'-l-3} \cdot 15\cdot O(4^{p-l'})+\sum_{l'=l+3}^{p-p_0-1}4\cdot 1^{l'-l-2} \cdot 15\cdot O(4^{p-l'})}_\text{from all the V to the bottom}.
	\end{eqnarray}
	Note that the first item in (\ref{eqn:SLR}) dominates and there are $4^l$ S-interactions in level $l$. The total complexity of low-rank approximation in the HRCM is in the order of
	\begin{equation}
		\sum_{l=0}^{p-p_0-1}O(4^{p-l-2}) = O(p\times 4^p) = O(N\log{N}).
	\end{equation}
		Combining Eqs. (\ref{eqn:SLR}) and (\ref{eqn:SD}), we claim that the complexity of the HRCM is $O(N\log{N})$.

	\section{Numerical results}
	
	In this section, we present the accuracy and efficiency of the proposed HRCM in 2D computations. For all the following simulations, we take $N = 4^p$ target/source points uniformly distributed in square domains of length $L$. In the random kernel compression Algorithm \ref{alg:svd}, the total numbers of sampled columns and rows are denoted as $c=r=K$, respectively.
	Numerical error, or the difference between direct multiplication,  is defined in sense of sample mean, i.e.
	\begin{equation}
		\text{Mean error} = \frac{1}{N_s}\sum_{i=1}^{N_s} \frac{\|({\bf A}-\Pi_{i}{\bf A}){\bf x}\|}{\|{\bf Ax}\|},
	\end{equation}
	where $\Pi_{i}{\bf A}$ is the $i$-th realization of the compressed matrix by the HRCM.
	
	\subsection{Accuracy and efficiency for well-separated sets}
	
	 First, we investigate the decay of singular values for the kernel matrix formed by the well-separated target and source points. The kernel function is taken as $\mathcal{K}({\bf r}_i, {\bf r}_j)=e^{-0.01R}/R$ with $L=8$, and the SVD of  relatively small matrices with $N = 1024$ are calculated, with diameter/distance ratios being $\eta=0.5, 0.36$, and 0.25. Logarithmic values  of the first 18 singular values for each case are  displayed  in Fig. \ref{fig:sigma} (a). It clearly shows that singular values of the matrix decay faster as the corresponding target and source sets are further away. For the fixed $\eta = 0.5$, approximated singular values from randomly sampled matrix ${\bf C}_r\in\K^{K\times K}$, with $K = 4^2, 4^3$ and $4^4$ are presented in Fig. \ref{fig:sigma} (b). The relative error with respect to the largest singular value is small enough after several singular values even for a very small amount of samples.
%
	\begin{figure}[th]
		\begin{center}%
			\begin{tabular}[c]{cc}%
				\includegraphics[width=0.48\textwidth]{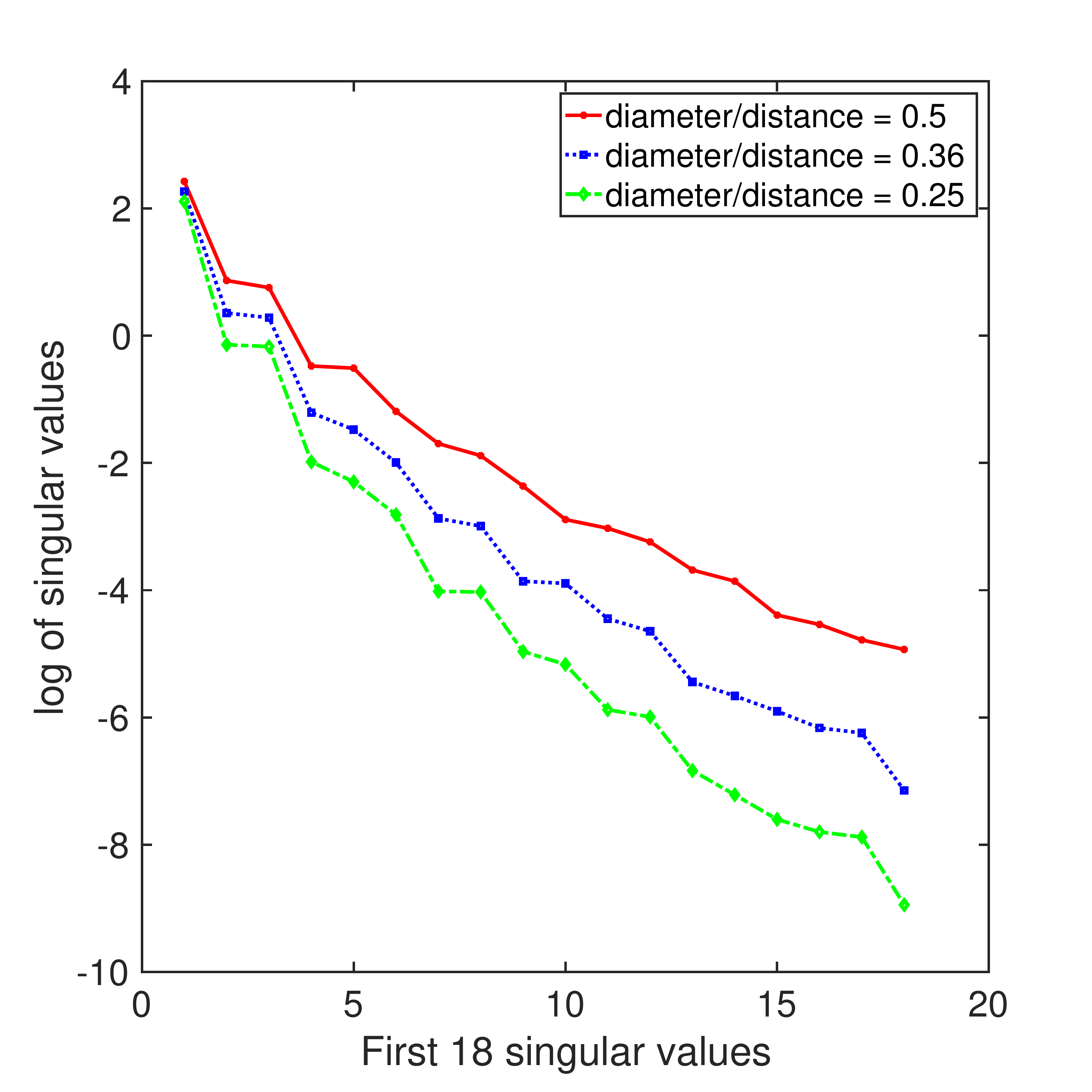} &
				\includegraphics[width=0.48\textwidth]{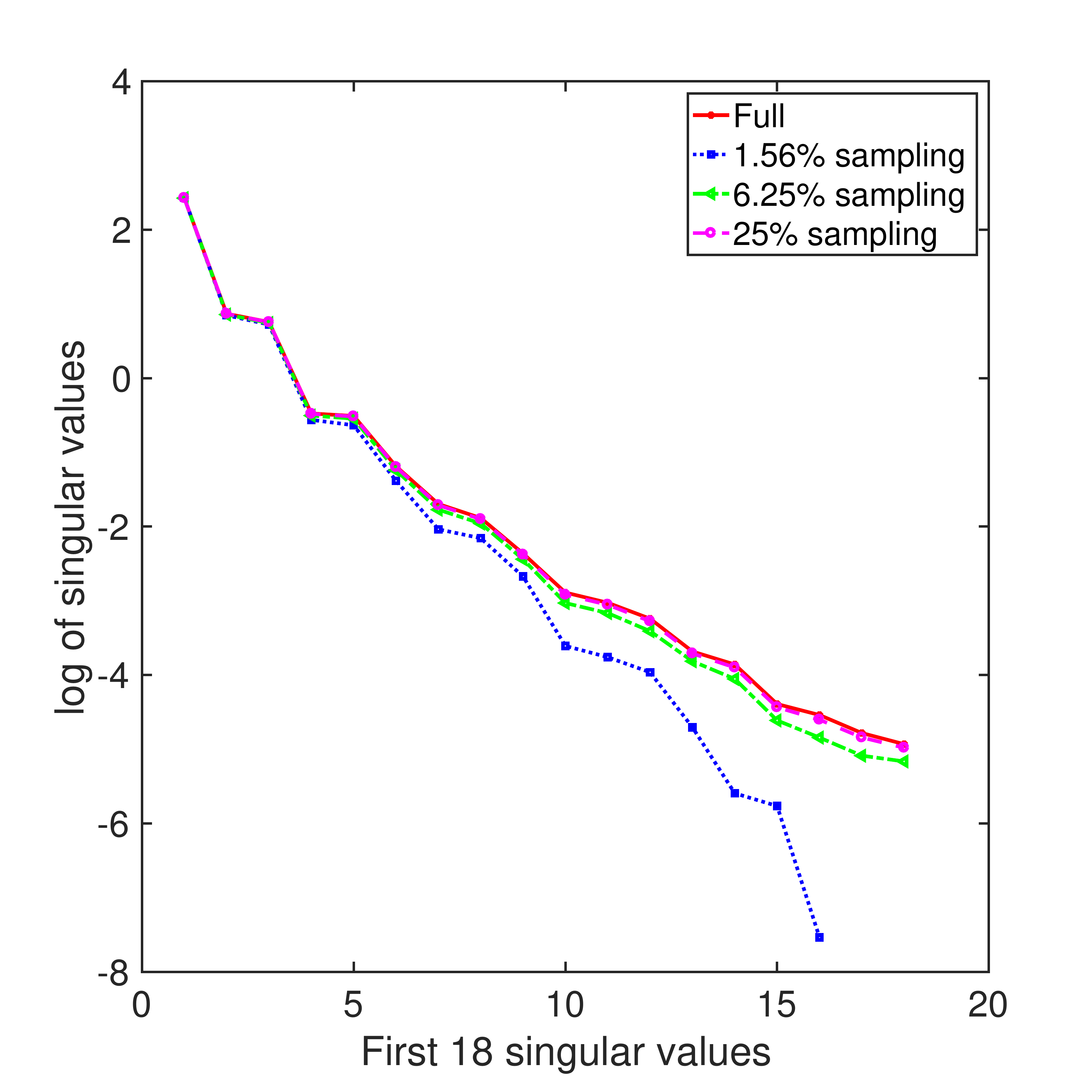}\\
				(a) & (b)
			\end{tabular}
		\end{center}
		\caption{(a) Singular values for matrices from well-separated points (1024 targets and 1024 source points) with various cluster diameter/distance ratios; (b) singular values comparison between different amounts of matrix samplings: $K = 16, 64, 256$ against $N = 1024$.}%
		\label{fig:sigma}%
	\end{figure}
	
	Next, we check the the algorithm accuracy. A total of $N$ target points and source points are uniformly assigned in two $8\times 8$ boxes with centers 16 units apart. Then, the direct multiplication (\ref{eqn:original}) and Algorithm \ref{alg:svd} are performed with parameter $c=r=K$ and $\epsilon = 1.0\times 10^{-8}$. For each comparison, sample mean and variance of errors are calculated with number of realization of HRCM $N_s=20$. Errors and variances for kernels $\mathcal{K}({\bf r}_i, {\bf r}_j)=\log{(R)}$ and $\mathcal{K}({\bf r}_i, {\bf r}_j)=\exp{(-0.01R)}/R$ are displayed in Tables \ref{table:kvsplog}-\ref{table:kvsp}, respectively,  with various $N$ and $K$ and $\eta=0.5$. We can clearly observe the convergence of the mean errors against $K$ in these tables.

	\begin{table}[ht]
		\caption{Errors  and variances for a pair of well-separated target/source point sets. Kernel function $\mathcal{K}({\bf r}_i, {\bf r}_j)=\log{(\sqrt{(x-x')^2+(y+y')^2})}-\log{(\sqrt{(x-x')^2+(y-y')^2})}$}%
		\label{table:kvsplog}%
		\centering
		\begin{tabular}[c]{c|l|l|l|l|l}\hline\hline
		 & $N = 1,024$ & $ N = 4096$ & $ N = 16, 384 $& $N = 65,536$ & $ N = 262,144$ \\\hline
		$K = 16$ &&&&\\
		Mean          & 2.79E-2 & 3.07E-2 & 3.51E-2 & 3.78E-2 & 4.01E-2\\
		Variance     & 3.58E-4 & 4.13E-4 & 5.38E-4 & 6.32E-4 & 6.95E-4\\
		\hline
		$K = 64$ & &   &  &   & \\
		Mean & 8.06E-3 & 8.54E-3 & 9.70E-3 & 9.84E-3 & 1.01E-2 \\
		Variance & 5.46E-6 & 5.89E-6 & 4.92E-6 & 6.27E-6 & 6.18E-6 \\
		\hline
		$K = 256$ &  &  &  & &   \\
		Mean & 2.25E-3 & 2.39E-3 & 2.52E-3 & 2.90E-3 & 2.75E-3 \\
		Variance & 4.76E-6 & 4.71E-6 & 5.37E-6 & 5.63E-6 & 5.86E-6 \\
		\hline\hline
		\end{tabular}
	\end{table}

	\begin{table}[ht]
		\caption{Errors  and variances for a pair of well-separated target/source point sets. Kernel function $\mathcal{K}({\bf r}_i, {\bf r}_j)=\exp{(-0.01R)}/R$}%
		\label{table:kvsp}%
		\centering
		\begin{tabular}[c]{c|l|l|l|l|l}\hline\hline
		 & $N = 1,024$ & $ N = 4096$ & $ N = 16, 384 $& $N = 65,536$ & $ N = 262,144$ \\\hline
		$K = 16$ &&&&\\
		Mean          & 2.67E-2 & 3.39E-2 & 3.07E-2 & 3.02E-2 & 3.51E-2\\
		Variance     & 7.51E-4 & 4.44E-4 & 6.28E-4 & 6.62E-4 & 9.95E-4\\
		\hline
		$K = 64$ & &   &  &   & \\
		Mean & 7.46E-3 & 7.58E-3 & 6.70E-3 & 8.51E-3 & 8.40E-3 \\
		Variance & 1.41E-5 & 2.78E-5 & 3.89E-5 & 4.17E-5 & 4.86E-5 \\
		\hline
		$K = 256$ &  &  &  & &   \\
		Mean & 1.62E-3 & 1.85E-3 & 1.92E-3 & 2.10E-3 & 2.30E-3 \\
		Variance & 1.76E-6 & 1.47E-6 & 1.37E-6 & 3.53E-6 & 3.53E-6 \\
		\hline\hline
		\end{tabular}
	\end{table}
	
	We conclude that based on the numerical results from Tables \ref{table:kvsplog}-\ref{table:kvsp} that, given the fixed diameter/distance ratio of the target/source point sets, the accuracy of the low-rank compression algorithm does not depend significantly on the total number $N$ but the sample number $K$.  It suggests that in computational practice, as long as two boxes are admissible clusters, it does not matter how many target/source points in them, the algorithm accuracy is purely controlled by the diameter/ration distance and number of samples.  	
	
	Table \ref{table:kvsptime} summarizes the corresponding computational time in seconds for the matrix-vector product, for direct computation and the low-rank compression method. If the target and source points are well-separated, the algorithm is very efficient and the computational time is linear both in sample size $K$ and matrix size $N$. CPU times for the two kernel functions are similar so only one of them is presented.

	\begin{figure}[th]
		\begin{center}%
				\includegraphics[width=0.6\textwidth]{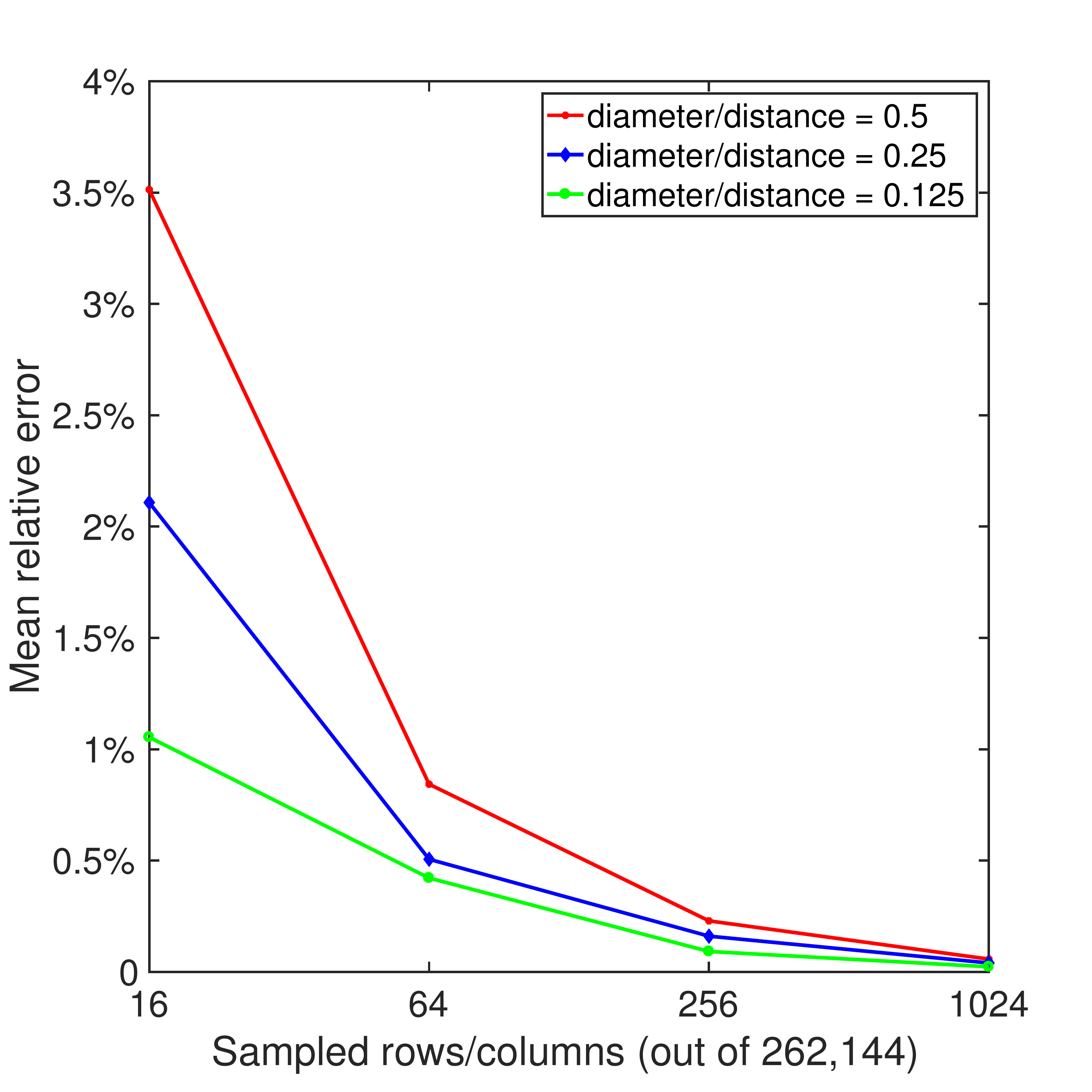}
		\end{center}
		\caption{Error of low-rank compression algorithm against diameter-distance ratio $\eta$.}%
		\label{fig:errnt}%
	\end{figure}
	
	Figure \ref{fig:errnt} shows the algorithm error against the diameter-distance ratios with $N = 262,144$ and different values of $K$. As expected, the relative error decays as $\eta$ increases. This graph is for kernel $\mathcal{K}({\bf r}_i, {\bf r}_j)=\exp{(-0.01R)}/R$, the one for kernel  $\mathcal{K}({\bf r}_i, {\bf r}_j)=\log{(R)}$ is similar.

	\begin{table}[ht]
		\caption{CPU time (second) comparison for well-separated target and source points.}%
		\label{table:kvsptime}%
		\centering
		\begin{tabular}[c]{c|l|l|l|l|l}\hline\hline
		 & $N = 1,024$ & $ N = 4096$ & $ N = 16, 384 $& $N = 65,536$ & $ N = 262,144$ \\\hline
		 Direct  & 0.047 & 0.75 & 12 & 204 & 3,264\\
		 \hline
		$K = 16$ & 0.01 & 0.039 & 0.16 & 0.625 & 2.5\\
		$K = 64$ & 0.04 & 0.16 & 0.66 & 2.6 & 12.0 \\
		$K = 256$ & 0.18 & 0.8 & 2.7 & 12.5 & 47.0 \\
		\hline\hline
		\end{tabular}
	\end{table}

	\subsection{Accuracy and efficiency for a single source and target set }
	Lastly, we test the accuracy and efficiency of the HRCM for target and source points in a  same set. Totally $N=4^p$ points with $ p = 6,7,8,9,10,11$ are uniformly distributed in the domain $[0, 8]\times [0,8]$. For best computation efficiency, we only present the results with $K=16$ and $64$. Note that it has been concluded that once $K$ is fixed, the accuracy of the low-rank compression algorithm for a pair of admissible cluster does not change too much regardless of number of points in them.
	
	The error for kernel $\mathcal{K}({\bf r}_i, {\bf r}_j)=\exp{(-0.01R)}/R$ with different $K$ and $N$ are summarized in Table \ref{table:accuracyHRC}. Note these values are generally smaller than those in Table \ref{table:kvsp}. Because Table \ref{table:kvsp} is for a single pair of  well-separated target/source sets with diameter-distance $\eta=0.5$. But  in the HRCM there exist a mixture of $\eta$ with  $\eta=0.5$ as the largest value. Additionally, similar convergence of error with respected to $K$ is shown in the table.
	\begin{table}[ht]
		\caption{Accuracy of the HRCM  for kernel summation with $\mathcal{K}({\bf r}_i, {\bf r}_j)=\exp{(-0.01R)}/R$}%
		\label{table:accuracyHRC}%
		\centering
		\begin{tabular}[c]{c|l|l|l|l}\hline\hline
		Matrix size & $ N =16,384$ & $N = 65,536$ & $ N = 262,144$ & $N  = 1,048,576$\\\hline
		$k=16$ &  &  &  &   \\
		Mean & 2.87E-3 & 3.32E-3 & 3.46E-3 & 3.53E-3 \\
		Variance & 6.82E-7 & 7.32E-7 & 7.65E-7 & 8.30E-7 \\
		\hline
		$k = 64$ &  &  & &  \\
		Mean & 6.09E-4 & 7.43E-4 & 6.26E-4 & 7.32E-4 \\
		Variance & 7.03E-8 & 6.49E-8 & 6.63E-8 & 7.56E-8 \\
		\hline\hline
		\end{tabular}
	\end{table}
	
	Next we check the error when the parameter $k$ takes a critical role in condition (\ref{eqn:assum}). We consider the Green's function for Helmholtz equation, $\mathcal{K}({\bf r}_i, {\bf r}_i) = \exp{(-ikR)/R}$. Errors and variances for this kernel with $k=0.25$ and $k=5$ are presented in Tables \ref{table:accuracyHRC2} and \ref{table:accuracyHRC3}, respectively.
	\begin{table}[ht]
		\caption{Accuracy of the HRCM  for kernel summation with $\mathcal{K}({\bf r}_i, {\bf r}_j)=\exp{(-ikR)}/R$, $k=0.25$}%
		\label{table:accuracyHRC2}%
		\centering
		\begin{tabular}[c]{c|l|l|l|l}\hline\hline
		Matrix size & $ N =16,384$ & $N = 65,536$ & $ N = 262,144$ & $N  = 1,048,576$\\\hline
		$K=16$ &   &   &  &   \\
		Mean & 2.56E-3 & 2.68E-3 & 2.71E-3 & 2.89E-3 \\
		Variance & 6.11E-7 & 3.56E-7 & 1.35E-7 & 1.01E-7 \\
		\hline
		$K = 64$ &  &   & &  \\
		Mean & 5.42E-4 & 5.51E-4 & 5.57E-4 & 5.89E-4 \\
		Variance & 1.43E-8 & 7.69E-8 & 2.18E-9 & 1.32E-9 \\
		\hline\hline
		\end{tabular}
	\end{table}
	
	\begin{table}[ht]
		\caption{Accuracy of the HRCM  for kernel summation with $\mathcal{K}({\bf r}_i, {\bf r}_j)=\exp{(-ikR)}/R$, $k=5$}%
		\label{table:accuracyHRC3}%
		\centering
		\begin{tabular}[c]{c|l|l|l|l}\hline\hline
		Matrix size & $ N =16,384$ & $N = 65,536$ & $ N = 262,144$ & $N  = 1,048,576$\\\hline
		$K=16$ &   &   &  &   \\
		Mean & 1.08E-2 & 1.38E-2 & 1.71E-2 & 1.98E-2 \\
		Variance & 4.55E-7 & 3.26E-7 & 2.45E-7 & 1.38E-7 \\
		\hline
		$K = 64$ &  &   & &  \\
		Mean & 2.87E-3 & 3.58E-3 & 4.53E-3 & 5.24E-3 \\
		Variance & 1.00E-7 & 3.61E-8 & 1.26E-8 & 3.32E-9 \\
		\hline\hline
		\end{tabular}
	\end{table}

		\begin{figure}[th]
		\begin{center}%
			\includegraphics[width=0.6\textwidth]{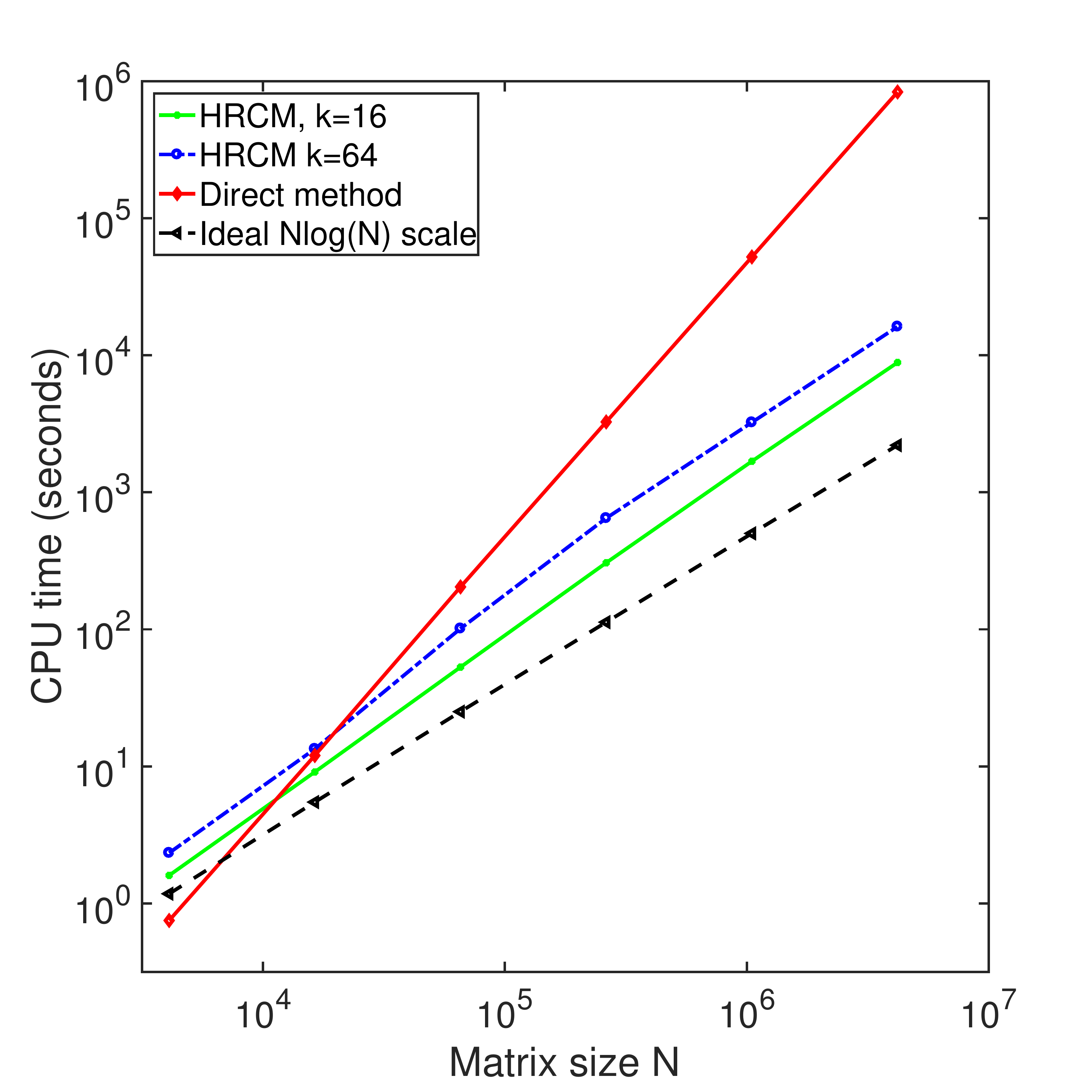}
		\end{center}
		\caption{CPU costs of HRCM with $K = 16$ and $K = 64$, for various matrix sizes. For comparison, the CPU time for direct method is shown in red and the ideal $N\log{(N)}$ curve is in black. }%
		\label{fig:timecost}%
	\end{figure}
	
	The efficiency of the HRCM is presented   in Fig. \ref{fig:timecost} as $\log$-$\log$ CPU time and matrix size $N$. For better comparison, the curves of CPU time for the direct method and an ideal $O(N\log{N})$ scale are also displayed. For HRCM with $K=16$ and $K=64$, the curves are almost parallel to the ideal $O(N\log{N})$ scale. Combining Fig. \ref{fig:timecost} and Tables \ref{table:accuracyHRC}-\ref{table:accuracyHRC3}, we can conclude that the break-even point of the HRCM comparing to the direct method with three or four digits in relative error is $N$ slightly larger than $10^4$. If higher accuracy is desired, one may have to increase number of $K$, hence the break-even point will be larger.

	\section{Conclusion and discussion}
	
	Kernel summation at large scale is a common challenge in a wide range of fields, from problems in computational sciences and engineering to statistical learnings. In this work, we have developed a novel hierarchical random compression method (HRCM) to tackle this common difficulty. The HRCM is a fast Monte-Carlo method that can reduce computation complexity from $O(N^2)$ to $O(N\log N)$ for a given accuracy. The method can be readily applied to iterative solver of linear systems resulting from  discretizing surface/volume integral equations of Poisson equation, Helmholtz equation or Maxwell equations, as well as fractional differential equations. It also applies to machine learning methods such as regression or classification for massive volume and high dimensional data.
	
	In designing HRCM, we first developed a random compression algorithm for kernel matrices resulting from far-field interactions, based on the fact that the interaction matrix from well-separated target and source points is of low-rank. Therefore,  we could sample a small number of columns and rows, independent of matrix sizes and only dependent on the separation distance between source and target locations, from the large-scale matrix, and then perform SVD on the small matrix, resulting in a low-rank approximation to the original matrix. A key factor in the HRCM is that a uniform sampling, implemented without cost of computing the usual sampling distribution based on the magnitude of sampled columns/rows, can yield a nearly optimal error in the low-rank approximation algorithm. HRCM is kernel-independent without the need for analytic formulaes of the kernels. Furthermore, an error bound of the algorithm with some assumption on kernel function was also provided in terms of the smoothness of the kernel, the number of samples and diameter-distance ratio of the well-separated sets.
	
	For general source and target configurations, we applied the concept of $\mathscr{H}$-matrix to hierarchically divide the whole matrix into logical block matrices, for which the developed low-rank compression algorithm can be applied if blocks correspond to a low-rank far field interactions at an appropriate scale, or they are divided further until direct summation is needed. Different from analytic or algebraic FMMs, the recursive structure nature of HRCM only execute an one-time, one way top-to-down path along the hierarchical tree structure: once a low-rank matrix is compressed, the whole block is removed from further consideration, and have no communications with the remaining entries of the whole kernel matrix.  As the HRCM combines the $\mathscr{H}$-matrix structure and low-rank compression algorithms, it has an $O(N\log N)$ computational complexity.
	
	Numerical simulations are provided for source and targets in two-dimensional (2D) geometry for several kernel functions, including 2D and 3D Green's function for Laplace equation, Poisson-Boltzmann equation and Helmholtz equation. In various cases, the mean relative errors of the HRCM against direct kernel summation show convergence in terms of number of samples and diameter-distance ratios. The computational cost was validated numerically as $O(N\log{N})$. Additionally, the break-even point with direct method is in the order of thousands, with three or four digit relative error.
	
	For future work, convergence rate of the HRCM is needed in terms of the number of realizations (i.e. $N_s$ ) and the rank $K$ parameter. Also, we will improve the performance of the HRCM in treating  high frequency wave problems for Helmholtz equations. As shown by our simulations, the mean error was significantly large when the wave number $k$ in the kernel is big. Simply increasing numbers of sampled columns or rows in the low-compression algorithm is one of the ways to handle the difficulty,  but may not be the best way. In addition, the HRCM will be extended to handled  data with even higher dimensions.

\section*{Acknowledgement}

The work was supported by US Army Research Office (Grant No. W911NF-17-1-0368)
and US National Science Foundation (Grant No. DMS-1802143).

\bibliographystyle{plain}
\bibliography{refs}

\end{document}